\documentclass[amstex,12pt]{amsart}
\usepackage{amssymb}
\usepackage{amsmath}
\usepackage{amscd}
\usepackage{mathrsfs}
\usepackage{graphicx}
\usepackage{latexsym}
\usepackage[all]{xy}

\theoremstyle{definition} 

\newtheorem*{definition}{Definition} 

\theoremstyle{plain}  

\newtheorem{theorem}{Theorem}[section]
\newtheorem{proposition}[theorem]{Proposition}
\newtheorem{corollary}[theorem]{Corollary}
\newtheorem{lemma}[theorem]{Lemma}

\DeclareMathOperator{\id}{id}

\DeclareMathOperator{\esssup}{esssup}
\DeclareMathOperator{\QC}{QC}
\DeclareMathOperator{\AC}{AC}
\DeclareMathOperator{\QS}{QS}

\DeclareMathOperator{\Sym}{Sym}
\DeclareMathOperator{\Mob}{\mbox{\rm{M\"ob}}}
\DeclareMathOperator{\Diff}{Diff}

\DeclareMathOperator{\Bel}{Bel}
\DeclareMathOperator{\Ael}{Ael}

\DeclareMathOperator{\D}{\mathbb D}

\DeclareMathOperator{\R}{\mathbb R}
\DeclareMathOperator{\N}{\mathbb N}

\DeclareMathOperator{\S1}{\mathbb S}
\DeclareMathOperator{\Chat}{\widehat {\mathbb C}}

\begin{document}

\title[The quotient Bers embedding]{Injectivity of the quotient Bers embedding of Teich\-m\"ul\-ler spaces}

\author[K. Matsuzaki]{Katsuhiko Matsuzaki}
\address{Department of Mathematics, School of Education, Waseda University,\endgraf
Shinjuku, Tokyo 169-8050, Japan}
\email{matsuzak@waseda.jp}

\thanks{This work was supported by JSPS KAKENHI 25287021.}

\begin{abstract}
The Bers embedding of the Teich\-m\"ul\-ler space is a homeomorphism into the Banach space of 
certain holomorphic automorphic forms. 
For a subspace of the universal Teich\-m\"ul\-ler space and its corresponding Banach subspace,
we consider whether the Bers embedding can project down between their quotient spaces.
If this is the case, it is called the quotient Bers embedding. 
Injectivity of the quotient Bers embedding is the main problem 
in this paper. Alternatively, we can describe this situation as
the universal Teich\-m\"ul\-ler space having an affine foliated structure induced by this subspace.
We give several examples of subspaces for which the injectivity holds true,
including the Teich\-m\"ul\-ler space of circle diffeomorphisms with H\"older continuous derivative. 
As an application, the regularity of conjugation between representations 
of a Fuchsian group into the group of circle diffeomorphisms is investigated.
\end{abstract}

\maketitle

\section{Introduction}\label{1}
The universal Teich\-m\"ul\-ler space $T$ is the ambient space of any other Teich\-m\"ul\-ler spaces.
An affine foliated structure of $T$ is induced by its certain subspace through the Bers embedding 
$\beta$ of $T$ into
the Banach space $B(\D^*)$ of hyperbolically bounded holomorphic quadratic automorphic forms on the disk $\D^*=\Chat-\overline{\D}$
in the Riemann sphere centered at the infinity. 
This was first investigated by Gardiner and Sullivan \cite{GS}
for the little Teich\-m\"ul\-ler subspace $T_0$, which consists of the asymptotically conformal elements of $T$.
This subspace is embedded by $\beta$ into the Banach subspace $B_0(\D^*)$ of $B(\D^*)$ consisting of
all elements vanishing at the boundary.
They proved that
the foliated structure of $T$ given by the right translations of $T_0$ in $T$
corresponds to the affine foliation of $B(\D^*)$ by the subspace $B_0(\D^*)$ 
under the Bers embedding $\beta$. 
In other words,
the Bers embedding is compatible with
the coset decompositions $T_0 \backslash T$ and $B_0(\D^*) \backslash B(\D^*)$.

The asymptotic Teich\-m\"ul\-ler space $AT$ was introduced in \cite{GS} as the quotient space $T_0 \backslash T$.
The compatibility of the Bers embedding $\beta$ with the coset decompositions as mentioned above
yields a well-defined quotient map 
$$
\widehat \beta:T_0 \backslash T \to B_0(\D^*) \backslash B(\D^*),
$$
by which the complex structure modeled on the quotient Banach space $B_0(\D^*) \backslash B(\D^*)$
was provided for $AT$. Later, the argument was simplified by showing that 
$\widehat \beta$ is also injective. This is due to Kahn (see
Gardiner and Lakic \cite[Section 16.8]{GL}).
Earle, Markovic and Saric \cite[Theorem 4]{EMS} generalized the injectivity of the quotient Bers embedding
for the Teich\-m\"ul\-ler space $T(\D/\Gamma)$ of a Riemann surface $\D/\Gamma$ for a Fuchsian group $\Gamma$ with respect to 
the corresponding little Teich\-m\"ul\-ler subspace $T_0(\D/\Gamma)$.

In this paper, we show other examples of affine foliated structures of the universal Teich\-m\"ul\-ler space $T$
ensuring the well-definedness and the injectivity of the quotient Bers embeddings. The subspaces of $T$ we handle here are the $p$-integrable
Teich\-m\"ul\-ler space $T^p$ (see Cui \cite{Cui}, Guo \cite{Guo}, Shen \cite{Sh}, Tang \cite{Tang}
and Yanagishita \cite{Yan}) and 
the Teich\-m\"ul\-ler space $T_0^{>0}$ of
circle diffeomorphisms of H\"older continuous derivative of an arbitrary exponent (see \cite{Mat0}, \cite{Mat3}).
The definitions of these Teich\-m\"ul\-ler spaces and the precise statements of our main theorems are 
given in Sections \ref{4} and \ref{5}, respectively. 
For the $2$-integrable
Teich\-m\"ul\-ler space $T^2$, Takhtajan and Teo \cite{TT} have shown
the well-definedness of the quotient Bers embedding of $T^2 \backslash T$, 
but the injectivity seems a new result.

The proofs of the injectivity mentioned above are based on a common argument. In Section \ref{3}, we summarize it as
a general principle. The injectivity of the quotient Bers embedding has been also proved in a different setting
and in a different method. See a recent work of Wei and Zinsmeister \cite{WZ}.
 
Affine foliated structures can be also defined on other Teich\-m\"ul\-ler spaces than $T$.
In Sections \ref{6}, we consider such situations, and in particular,
we prove the affine foliated structure of $T_0^{>0}$ induced by the Teich\-m\"ul\-ler space $T_0^\alpha$ of
circle diffeomorphisms of $\alpha$-H\"older continuous derivative for $\alpha \in (0,1)$. 
We have obtained in \cite{Mat2} a complex structure on 
$T_0^\alpha$ modeled on a certain Banach space via the Bers embedding. Our result in particular shows that
$T_0^{>0}$ admits such a Banach manifold structure with the decomposition into mutually disjoint but equivalent components,
and each component corresponds injectively to an affine subspace of the Bers embedding.

As an application of one of our main theorems,
we can represent the deformation space $DT(\Gamma)$ of
a Fuchsian group $\Gamma \subset \Mob(\S1) \cong {\rm PSL}(2,\R)$ 
in the group $\Diff_+^{>1}(\mathbb S)$ of all circle diffeomorphisms with
H\"older continuous derivatives of any exponent
as a subspace of the Teich\-m\"ul\-ler space $AT(\Gamma)$ of $\Gamma$-invariant symmetric structures on $\S1$.
Here, $AT(\Gamma)$ is the closed subspace of $AT$ consisting of all elements of $AT$ fixed by 
the action of every $\gamma \in \Gamma$.
This space was studied in \cite{Mat6}. To show the injectivity of $DT(\Gamma) \to AT(\Gamma)$,
we also need a rigidity theorem for the representation of $\Gamma$ in $\Diff_+^{1+\alpha}(\mathbb S)$
given in \cite{Mat7}. Applying this theorem,
we finally prove in Section \ref{7} that if two representations of $\Gamma$ in $\Diff_+^{r}(\mathbb S)$ for $r>1$ are conjugate by
a symmetric homeomorphism $f$ representing an element of $T_0$, then $f$ actually belongs to $\Diff_+^{r}(\mathbb S)$.

\section{Preliminaries and background results}\label{2}
An orientation-preserving homeomorphism $w$ of a domain in the complex plane $\mathbb C$
is said to be {\it quasiconformal} if partial derivatives $\partial w$ and $\bar \partial w$ in
the distribution sense exist and
if the {\it complex dilatation}
$\mu_w(z)=\bar \partial w(z)/\partial w(z)$ satisfies $\Vert \mu_w \Vert_\infty <1$.
Let
$$
\Bel(\D)=\{\mu \in L^\infty(\mathbb \D) \mid \Vert \mu \Vert_\infty <1\}
$$ 
be the space of such measurable functions on the unit disk $\D$, which are called {\it Beltrami coefficients}.
We denote the group of all quasiconformal self-homeomorphisms of $\D$ by $\QC(\D)$.
By the measurable Riemann mapping theorem (see \cite{Ah0}), 
for every $\mu \in \Bel(\D)$, there is $w \in \QC(\D)$ satisfying $\mu_w=\mu$
uniquely up to the post-composition of elements of $\Mob(\D) \cong {\rm PSL}(2,\mathbb R)$, the group of 
all M\"obius transformations of $\D$. This gives the identification
$$
\Mob(\D)\backslash \QC(\D) \cong \Bel(\mathbb \D). 
$$

Every $w \in \QC(\D)$ extends continuously to a {\it quasisymmetric} self-homeomorphism of $\S1=\partial \D$.
Let $\QS$ be the group of all quasisymmetric self-homeomorphisms of $\S1$. 
We denote the boundary extension map by
$$
q:\QC(\D) \to \QS,
$$
which is a surjective homomorphism.
The {\it universal Teich\-m\"ul\-ler space} is defined by
$$
T=\Mob(\S1) \backslash \QS, 
$$
where $\Mob(\S1)=q(\Mob(\D))$.
Then, $q$ induces the Teich\-m\"ul\-ler projection $\pi:\Bel(\D) \to T$.
The quotient topology of $T$ is induced from the norm on $\Bel(\D)$ by $\pi$. In fact, the Teich\-m\"ul\-ler distance can be
defined by using the hyperbolic distance on $\Bel(\D)$.

For every $\mu \in \Bel(\D)$, we
extend it to a Beltrami coefficient $\widehat \mu$ on the Riemann sphere $\widehat{\mathbb C}$
by setting $\widehat \mu(z) \equiv 0$ for $z \in \mathbb D^*=\widehat{\mathbb C}-\overline{\mathbb D}$.
We denote a quasiconformal homeomorphism of $\Chat$ with the complex dilatation $\widehat \mu$ by
$f_\mu$.
The measurable Riemann mapping theorem guarantees the existence of such $f_\mu$
and the uniqueness of $f_\mu$ up to the post-composition of M\"obius transformations of $\widehat{\mathbb C}$.
We take the Schwarzian derivative 
$S_{f_\mu}:\mathbb D^* \to \widehat{\mathbb C}$ of 
the conformal homeomorphism $f_\mu|_{\mathbb D^*}$, which 
parametrizes the complex projective structures on $\D^*$.
By the Nehari-Kraus theorem, 
$S_{f_\mu}$ belongs to the complex Banach space of hyperbolically bounded holomorphic quadratic automorphic forms
$$
B(\D^*)=\{\varphi(z)dz^2 \mid \Vert \varphi \Vert_\infty :=\sup_{z \in \D^*} \rho^{-2}_{\D^*}(z)|\varphi(z)|<\infty\},
$$
where $\rho_{\D^*}(z)=2/(|z|^2-1)$ is the hyperbolic density on $\D^*$.
By this correspondence $\mu \mapsto S_{f_\mu}$, a map
$$
\Phi:\Bel(\D) \to B(\D^*)
$$
is defined to be a holomorphic split submersion, which is called the {\it Bers projection} (onto the image).

For the Teich\-m\"ul\-ler projection $\pi:\Bel(\D) \to T$ and
the Bers projection $\Phi:\Bel(\D) \to B(\D^*)$, we can show that
$\Phi \circ \pi^{-1}$ is well-defined and injective, which defines 
a map $\beta:T \to B(\D^*)$ called the {\it Bers embedding}.
In fact, $\beta$ is a homeomorphism onto the image $\beta(T)=\Phi(\Bel(\D))$ and $\beta(T)$ is a bounded domain in $B(\D^*)$.
This provides a complex Banach manifold structure for $T$.

There is a global continuous section for the Teich\-m\"ul\-ler projection $\pi:\Bel(\D) \to T$.
This is defined by a canonical quasiconformal extension $e:\QS \to \QC(\D)$
for each quasisymmetric self-homeomorphism $g$ of $\S1$. Douady and Earle \cite{DE} introduced
the {\it barycentric extension} $e_{\rm DE}:\QS \to \QC(\D)$
having the {\it conformal naturality} 
$$
e_{\rm DE}(\phi_1 \circ g \circ \phi_2)=e_{\rm DE}(\phi_1) \circ e_{\rm DE}(g) \circ e_{\rm DE}(\phi_2)
$$
for any $\phi_1, \phi_2 \in \Mob(\S1)$ and any $g \in \QS$, where $e_{\rm DE}(\phi_1)$ and $e_{\rm DE}(\phi_2)$ are in
$\Mob(\D)$.
Taking the quotient by $\Mob(\S1)=\Mob(\D)$ in both sides, 
we have a continuous map $s:T \to \Bel(\D)$ such that
$\pi \circ s=\id_T$. 
The existence of a global continuous section implies that $T$ is contractible.
Let $\sigma: \Bel(\D) \to \Bel(\D)$ be defined by the correspondence of
$\mu$ to $s(\pi(\mu))$ for the section $s$. 
The image $\sigma(\Bel(\D))$ is the set of all Beltrami coefficients obtained by the barycentric extension.

For any $\nu \in \Bel(\D)$, let $f^\nu \in \QC(\D)$ be a normalized element having 
the complex dilatation $\nu$, where
the {\it normalization} is given by fixing three boundary points $1$, $i$ and $-1$ on $\S1$.
The subgroup of $\QC(\D)$ consisting of all normalized elements is denoted by $\QC_*(\D)$.
Applying this normalization,
we can define a group structure on $\Bel(\D)$ as follows.
For any $\nu_1, \nu_2 \in \Bel(\D)$, let
$\nu_1 \ast \nu_2$ be the complex dilatation of the composition $f^{\nu_1} \circ f^{\nu_2}$.
Then, $\Bel(\D)$ is a group with this operation $\ast$. In other words,
by the identification of $\Bel(\D)$ with $\QC_*(\D)$, we regard $\Bel(\D)$ as a subgroup of $\QC(\D)$.
We denote the inverse element of $\nu \in \Bel(\D)$ by $\nu^{-1}$, which is the complex dilatation of $(f^\nu)^{-1}$.
The chain rule of partial differentials yields a formula
$$
\nu_1 \ast \nu_2^{-1}(\zeta)=\frac{\nu_1(z)-\nu_2(z)}{1-\overline{\nu_2(z)}\nu_1(z)}\cdot\frac{\partial f^{\nu_2}(z)}{\overline{\partial f^{\nu_2}(z)}}
\qquad(\zeta=f^{\nu_2}(z)).
$$
Each $\nu\in \Bel(\D)$ induces the right translation $r_\nu:\Bel(\D) \to \Bel(\D)$ by
$\mu \mapsto \mu \ast \nu^{-1}$. 
By the above formula, we see that $r_\nu$ and $(r_\nu)^{-1}=r_{\nu^{-1}}$ are continuous, and 
hence $r_\nu$ is a homeomorphism 
of $\Bel(\D)$. In fact, this is a biholomorphic automorphism of $\Bel(\D)$.

For the base point $[\id]$ of $T=\Mob(\S1) \backslash \QS$, the inverse image of the Teich\-m\"ul\-ler projection
$$
\pi^{-1}([\id])=\{\nu \in \Bel(\D) \mid q(f^{\nu})=\id\}
$$
is a normal subgroup of $\Bel(\D)$
since $q:\QC(\D) \to \QS$ is a homomorphism.
Having $T=\Bel(\D)/\pi^{-1}([\id])$, we see that $T$ has a group structure with the operation $\ast$ 
defined by $\pi(\nu_1) \ast \pi(\nu_2)=\pi(\nu_1 \ast \nu_2)$. 
The projection of the right translation $r_\nu:\Bel(\D) \to \Bel(\D)$ under $\pi$ yields a well-defined map
$R_{\pi(\nu)}:T \to T$ by 
$$
\pi(\mu) \mapsto \pi(\mu \ast \nu^{-1})=\pi(\mu)\ast \pi(\nu)^{-1}.
$$
In this way, we have the base point change $R_\tau:T \to T$ sending $\tau$ to
$[\id]$ for every $\tau \in T$.
Alternatively, 
$R_\tau:T \to T$ is defined by $[g] \mapsto [g \circ f^{-1}]$ for $\tau =[f] \in T$. 
We see that the base point change $R_\tau$ is also a biholomorphic automorphism of $T$.

Every element $\gamma \in \Mob(\S1)$ acts on $B(\D^*)$ linear isometrically through the Bers embedding $\beta$.
This means that, for any point $[f] \in T$ with $\beta([f])=\varphi \in \beta(T)$, the Bers embedding $\beta(\gamma^*[f])$ of 
$\gamma^*[f]:=[f \circ \gamma]$
is represented by
$$
(\gamma^*\varphi)(z)=\varphi(\gamma(z))\gamma'(z)^2,
$$
where we regard $\gamma$ as the element of $\Mob(\D^*)$,
the group of 
all M\"obius transformations of $\D^*$. Namely, 
$\gamma^*\varphi$ is the pull-back of $\varphi$ by $\gamma$ as a quadratic automorphic form.
Clearly, this action extends to $B(\D^*)$ and satisfies $\Vert \gamma^*\varphi \Vert_\infty=\Vert \varphi \Vert_\infty$.

A quasiconformal self-homeomorphism $w \in \QC(\D)$ is called {\it asymptotically conformal} if 
its complex dilatation vanishes at the boundary, that is, $\mu_w(z) \to 0$ as $|z| \to 1$.
The subspace of $\Bel(\D)$ consisting of all Beltrami coefficients vanishing at the boundary is denoted by $\Bel_0(\D)$ and
the subgroup of $\QC(\D)$ consisting of all asymptotically conformal self-homeomorphisms of $\D$ is denoted by
$\AC(\D)$. Every $w \in \AC(\D)$ extends continuously to a {\it symmetric} self-homeomorphism of $\S1$.
The group of all symmetric self-homeomorphisms of $\S1$ is denoted by
$\Sym$. Then, the restriction of the boundary extension to $\AC(\D)$ gives a surjective homomorphism
$q:\AC(\D) \to \Sym$. 

Gardiner and Sullivan \cite{GS} studied the {\it asymptotic Teich\-m\"ul\-ler space} defined by
$$
AT=\Sym \backslash \QS,
$$
and the {\it little universal Teich\-m\"ul\-ler space} defined by
$$
T_0=\Mob(\S1) \backslash \Sym=\pi(\Bel_0(\D)). 
$$
They introduced $\Sym$ as a particular topological subgroup of $\QS$.
The characterizations of symmetric self-homeomorphisms by 
the Bers embedding of $T_0$ 
was also given. The Banach subspace of $B(\D^*)$ consisting of the elements of vanishing at the boundary is denoted by
$$
B_0(\D^*)=\{\varphi \in B(\D^*) \mid \lim_{|z| \to 1} \rho^{-2}_{\D^*}(z)|\varphi(z)|=0\}.
$$

\begin{proposition}\label{gs}
For a quasisymmetric homeomorphism $g \in \QS$, the following conditions are equivalent:
$(1)$ $g \in \Sym$; $(2)$ $s([g]) \in \Bel_0(\D)$; $(3)$ $\beta([g]) \in B_0(\D^*)$.
\end{proposition}

We refer to Earle, Markovic and Saric \cite{EMS} for the barycentric extension. In particular, we see that 
$$
\beta(T_0)=\Phi(\Bel_0(\D))=\beta(T) \cap B_0(\D^*).
$$ 

Next, we consider how $T_0$ is mapped into $T$ by the base point change $R_\tau:T \to T$ for $\tau \in T$.
Since $R_\tau$ is a biholomorphic automorphism of $T$, $T_0$ is mapped biholomorphically onto
the image $R_\tau(T_0)$. 
We recall that $T_0$ is a subgroup of $(T,\ast)$ and $R_\tau$ is defined by
$R_\tau(\tau')=\tau' \ast \tau^{-1}$ for every $\tau' \in T$.
Then, the coset decomposition of $T$ by the subgroup $T_0$ is exactly 
the disjoint union 
$$
T=\bigsqcup_{[\tau] \in T_0 \backslash T} R_\tau(T_0)
$$
of mutually biholomorphically equivalent subspaces.

Moreover, we find that the image of the decomposition $T=\bigsqcup_{[\tau] \in T_0 \backslash T} R_\tau(T_0)$ under the 
Bers embedding $\beta:T \to B(\D^*)$ corresponds to the foliation of $\beta(T)$ by
the family of Banach affine subspaces $\{B_0(\D^*)+\psi\}_{[\psi] \in B_0(\D^*) \backslash B(\D^*)}$.
This compatibility can be formulated as the following theorem.

\begin{theorem}\label{B_0-B}
For each $\nu \in \Bel(\D)$, let $\psi=\Phi(\nu) \in B(\D^*)$. Then,
$$
\Phi \circ r_\nu^{-1}(\Bel_0(\D))=\beta(T) \cap \{B_0(\D^*)+\psi\}.
$$
Hence, $\beta \circ R_{\tau}^{-1}(T_0)=\beta(T) \cap \{B_0(\D^*)+\beta(\tau)\}$ for every
$\tau \in T$.
\end{theorem}

The fact that $\beta(T) \cap \{B_0(\D^*)+\psi\}$ contains $\Phi \circ r_\nu^{-1}(\Bel_0(\D))$
was proved in \cite{GS}. The converse inclusion is due to Kahn (see also Gardiner and Lakic \cite[Section 16.8]{GL}).
As a consequence from a property of the conformally natural section, 
a different proof of Theorem \ref{B_0-B} by using the barycentric extension,
which is also valid taking the action of a Fuchsian group into account,
can be also obtained from the arguments in \cite{EMS}.

By this result, we have the decomposition of the Bers embedding as
$$
\beta(T)=\bigsqcup_{[\tau] \in T_0 \backslash T} \beta \circ R_\tau^{-1}(T_0)
=\bigsqcup_{[\psi] \in B_0(\D^*) \backslash B(\D^*)}\beta(T) \cap (B_0(\D^*)+\psi).
$$
Each component $\beta(T) \cap (B_0(\D^*)+\psi)$ is biholomorphically equivalent to $T_0 \cong \beta(T_0)$. 
We call this decomposition of $T \cong \beta(T)$ the {\it affine foliated structure} of $T$ induced by $T_0$.

We consider the embedding of the quotient Teich\-m\"ul\-ler space $AT=T_0 \backslash T$ into 
the quotient Banach space $B_0(\D^*) \backslash B(\D^*)$. In the equation of Theorem \ref{B_0-B},
when we have the inclusion $\subset$, we see that the quotient map 
$\widehat \beta: T_0 \backslash T \to B_0(\D^*) \backslash B(\D^*)$ is well-defined.
This map is called the {\it quotient Bers embedding}. By showing that $\widehat \beta$ is a local homeomorphism,
a complex structure on $AT$ modeled on $B_0(\D^*) \backslash B(\D^*)$ was given in \cite{GS}.
The converse inclusion $\supset$ in Theorem \ref{B_0-B} implies the stronger result that $\widehat \beta$ is a homeomorphism onto the image,
and in particular $\widehat \beta$ is injective.

\begin{corollary}
The quotient Bers embedding $\widehat \beta:T_0 \backslash T \to B_0(\D^*) \backslash B(\D^*)$
is well-defined to be a homeomorphism onto the image.
\end{corollary}

\section{General principle}\label{3}

In this section, we prepare a basic argument to prove the injectivity of
the quotient Bers embedding.
This is carried out based on Theorem \ref{B_0-B}. We keep using the same notations throughout this section.

We fix $\psi=\Phi(\nu) \in \beta(T) \subset B(\D^*)$ for any $\nu \in \Bel(\D)$, and
take $\varphi \in B_0(\D^*)$ such that $\varphi +\psi \in \beta(T)$.
For a quasiconformal homeomorphism $f_\nu$ of $\Chat$ that is conformal on $\D^*$, we
set $\Omega=f_\nu(\D)$ and $\Omega^*=f_\nu(\D^*)$. Under these circumstances,
Theorem \ref{B_0-B} implies the following.

\begin{proposition}\label{basic}
There exists 
a quasiconformal homeomorphism $\widehat f:\Chat \to \Chat$ conformal on $\Omega^*$
and asymptotically conformal on $\Omega$ such that $S_{\widehat f \circ f_\nu|_{\D^*}}=\varphi +\psi$.
\end{proposition}

\begin{proof}
By Theorem \ref{B_0-B}, there is some $\mu \in \Bel_0(\D)$ such that $\Phi \circ r_\nu^{-1}(\mu)=\Phi(\mu \ast \nu)=\varphi +\psi$.
We consider $f_{\mu \ast \nu}$, the quasiconformal self-homeomorphism of $\Chat$ conformal on $\D^*$ and
quasiconformal on $\D$ with the complex dilatation $\mu \ast \nu$. Then, we define $\widehat f=f_{\mu \ast \nu} \circ f_\nu^{-1}$,
which satisfies $S_{\widehat f \circ f_\nu|_{\D^*}}=S_{f_{\mu \ast \nu}|_{\D^*}}=\varphi +\psi$.
Moreover, $\widehat f$ is conformal on $\Omega^*$ and quasiconformal on $\Omega$ with the complex dilatation
$(g_\nu^*\mu)(z)=\mu(g_\nu(z))\overline{g'_\nu(z)}/g'_\nu(z)$, where $g_\nu=f^\nu \circ f_\nu^{-1}:\Omega \to \D$ is the conformal
homeomorphism (Riemann map).
Hence, $\widehat f$ is asymptotically conformal on $\Omega$.
\end{proof}

The complex dilatation $\widehat \mu$ of
$\widehat f$ on $\Omega$ vanishes at the boundary $\partial \Omega$. In particular, we can choose a
compact subset $\Omega_0 \subset \Omega$ such that
$$
3\, \Vert \widehat \mu |_{\Omega-\Omega_0} \Vert_\infty < \delta(\Vert \nu \Vert_\infty),
$$
where $\delta(\Vert \nu \Vert_\infty)>0$ is a constant given later 
in Lemma \ref{reflection} depending only on $\Vert \nu \Vert_\infty$.

We decompose $\widehat f$ into $\widehat f_0 \circ \widehat f_1$ as follows.
The quasiconformal homeomorphism $\widehat f_1:\Chat \to \Chat$ is chosen so that its complex dilatation coincides with
$\widehat \mu$ on $\Omega-\Omega_0$ and zero elsewhere. Then, 
$\widehat f_0$ is defined to be $\widehat f \circ \widehat f_1^{-1}$, whose complex dilatation has
a support on the compact subset $\widehat f_1(\Omega_0) \subset \widehat f_1(\Omega)$.
We take $\varphi_1 \in B(\D^*)$ so that 
$$
S_{\widehat f_1 \circ f_\nu|_{\D^*}}=\varphi_1+\psi.
$$
This satisfies $\Vert \varphi_1 \Vert_\infty < \delta(\Vert \nu \Vert_\infty)$. Indeed,
\begin{align*}
\rho^{-2}_{\D^*}(z)|\varphi_1(z)|&
=\rho^{-2}_{\D^*}(z)|S_{\widehat f_1 \circ f_\nu|_{\D^*}}(z)-S_{f_\nu|_{\D^*}}(z)|\\
&=\rho^{-2}_{\Omega^*}(\zeta)|S_{\widehat f_1|_{\Omega^*}}(\zeta)|
\end{align*}
for $\zeta=f_\nu(z)$ and this is bounded by $3 \Vert \widehat \mu|_{\Omega-\Omega_0} \Vert_\infty$
(see \cite[Theorem II.3.2]{Leh}).

We utilize a 
local section for the Bers projection $\Phi:\Bel(\D) \to \beta(T) \cap B(\D^*)$,
which is a generalization of the Ahlfors--Weill section defined in a neighborhood of the origin, and was
constructed by using a quasiconformal reflection originally due to
Ahlfors \cite{Ah}. This was improved later with the aid of the barycentric extension
by Earle and Nag \cite{EN} to hold compatibility
with the action of M\"obius transformations. 
The following assertion can be also proved by the arguments in \cite[Section 14.4]{GL} and \cite[Section II.4.2]{Leh}.

\begin{lemma}\label{reflection}
Let $f_\nu:\Chat \to \Chat$ be the quasiconformal homeomorphism with 
complex dilatation $\nu \in \sigma(\Bel(\D))$ obtained by the barycentric extension, 
which is conformal on $\D^*$ with $S_{f_\nu|_{\D^*}}=\psi$.
Then, there exists a constant
$\delta=\delta(\Vert \nu \Vert_\infty) >0$ depending only on $\Vert \nu \Vert_\infty$ such that
for every $\varphi \in B(\D^*)$ with $\Vert \varphi \Vert_\infty < \delta$,
there is a quasiconformal homeomorphism $\widehat f:\Chat \to \Chat$ conformal on $\Omega^*$ such that
$S_{\widehat f \circ f_\nu|_{\D^*}}=\varphi +\psi$ and the complex dilatation $\widehat \mu$ of $\widehat f$
on $\Omega$ satisfies
$$
|\widehat \mu(f_\nu(z))| \leq \frac{1}{\delta}\rho_{\D^*}^{-2}(z^*)|\varphi(z^*)|
$$
for every $z \in \D$. Here, $z^*=(\bar z)^{-1}$ is the reflection of $z \in \D$ with respect to $\S1$. 
\end{lemma}

We apply this lemma to the previous $\nu$ with $\psi=\Phi(\nu)$ and $\varphi_1$
with $S_{\widehat f_1 \circ f_\nu|_{\D^*}}=\varphi_1+\psi$.
We may always assume that $\nu$ is obtained by the barycentric extension.
Replacing $\widehat f_1$ with the quasiconformal homeomorphism obtained in Lemma \ref{reflection},
we can further assume that the complex dilatation $\widehat \mu_1$ of $\widehat f_1$ satisfies
the above inequality. We remark that $\widehat f_1|_{\Omega^*}$ does not change by this replacement. We use the same 
$\widehat f_0$ as before; correspondingly, $\widehat f|_\Omega$ changes but $\widehat f|_{\Omega^*}$ does not.

Having $f_\nu$, $\widehat f_1$ and $\widehat f_0$ already,
we take the normalized quasiconformal homeomorphisms $f^\nu:\D \to \D$,
$f_1:\D \to \D$ and $f_0:\D \to \D$, and the
conformal homeomorphisms (Riemann mappings) $g_\nu:\Omega \to \D$, $g_1:\widehat f_1(\Omega) \to \D$ and
$g:\widehat f(\Omega) \to \D$ so that the following commutative diagram holds:
$$ 
\xymatrix{
\D \ar[rr]^{f^\nu} \ar[rrd]_{f_\nu} & & \D \ar[rr]^{f_1} & & \D \ar[rr]^{f_0} & & \D \\
& & \Omega \ar[rr]_{\widehat f_1} \ar[u]_{g_\nu} & &\widehat f_1(\Omega) \ar[rr]_{\widehat f_0} \ar[u]_{g_1} & &\widehat f(\Omega) \ar[u]_{g}\ .
}
$$
We note that $g_\nu$, $g_1$ and $g$ are uniquely determined.
Set $f=f_0 \circ f_1$. Then, the complex dilatation of $\widehat f \circ f_\nu$ on $\D$ coincides with
that of $f \circ f^\nu$. Hence, its image under the Bers projection $\Phi$ is $\varphi +\psi$.
This is also true for $f_1$ and $\varphi_1$ instead of $f$ and $\varphi$.

We consider $\varphi-\varphi_1=(\varphi+\psi)-(\varphi_1+\psi)$ for $z \in \D^*$, which is equal to
\begin{align*}
S_{\widehat f \circ f_\nu|_{\D^*}}(z)-S_{\widehat f_1 \circ f_\nu|_{\D^*}}(z)&=
S_{\widehat f_0 \circ \widehat f_1 \circ f_\nu|_{\D^*}}(z)-S_{\widehat f_1 \circ f_\nu|_{\D^*}}(z).
\end{align*}
The complex dilatation $\mu_0$ of $f_0$ is equal to the push-forward $(g_1)_* \widehat \mu_{0}$ of
the complex dilatation $\widehat \mu_0$ of $\widehat f_0$ by the conformal homeomorphism $g_1$. 
Since the support of $\widehat \mu_0$ is on the compact subset of
$\widehat f_1(\Omega)$, the support of $\mu_0$ is on a compact subset of $\D$.
Similarly, the complex dilatation $\mu_1$ of $f_1$ is equal to $(g_\nu)_* \widehat \mu_{1}$.
Then, we see that $\varphi-\varphi_1$ coincides with $\Phi(\mu_0 \ast \mu_1 \ast \nu)-\Phi(\mu_1 \ast \nu)$.

The above arguments are summarized as follows.

\begin{proposition}\label{summary}
Let $f_\nu:\Chat \to \Chat$ be a quasiconformal homeomorphism with 
complex dilatation $\nu \in \sigma(\Bel(\D))$, which is conformal on $\D^*$ with $S_{f_\nu|_{\D^*}}=\psi$.
Let $\widehat f: \Chat \to \Chat$ be a quasiconformal homeomorphism
with complex dilatation $\widehat\mu$ on $\Omega=f_\nu(\D)$ vanishing at the boundary
that is conformal on $\Omega^*=f_\nu(\D^*)$ with
$S_{\widehat f \circ f_\nu|_{\D^*}}=\varphi +\psi$. 
Then, $\widehat f$ is decomposed into two quasiconformal homeomorphisms $\widehat f_0$ and 
$\widehat f_1$ of $\Chat$ with $\widehat f=\widehat f_0 \circ \widehat f_1$, where
$\widehat f_1$ is conformal on $\Omega^*$ with $S_{\widehat f_1 \circ f_\nu|_{\D^*}}=\varphi_1+\psi$,
satisfying the following properties:
\begin{enumerate}
\item
the complex dilatation $\widehat \mu_1$ of $\widehat f_1$ on $\Omega$
satisfies
$$
|\widehat \mu_1 \circ f_\nu(z)| \leq \frac{1}{\delta}\,\rho_{\D^*}^{-2}(z^*)|\varphi_1(z^*)|
$$
for some $\delta>0$ and for every $z \in \D$;
\item
the support of the complex dilatation $\mu_0$ of the normalized quasiconformal homeomorphism 
$f_0:\D \to \D$, which is conformally conjugate to
$\widehat f_0:\widehat f_1(\Omega) \to \widehat f(\Omega)$, is contained in a compact subset of $\D$;
\item
for the complex dilatation $\mu_1$ of the normalized quasiconformal homeomorphism 
$f_1:\D \to \D$, which is conformally conjugate to $\widehat f_1:\Omega \to \widehat f_1(\Omega)$,
we have
$$
\varphi-\varphi_1=\Phi(\mu_0 \ast \mu_1 \ast \nu)-\Phi(\mu_1 \ast \nu).
$$
\end{enumerate}
\end{proposition}

In the remainder of this section, 
we state two results, which provides a foundation for the argument on the affine foliated structure given by
the Bers embedding. 

\begin{proposition}\label{base1}
Let $f_\mu$ and $f_\nu$ be the quasiconformal homeomorphisms of $\Chat$
that are conformal on $\D^*$ and have complex dilatations $\mu$ and $\nu$ respectively on $\D$.
Let $\Omega=f_\nu(\D)$ and $\Omega^*=f_\nu(\D^*)$. Then,
\begin{align*}
|S_{f_\mu \circ f_\nu^{-1}|_{\Omega^*}}(\zeta)| 
&\leq \frac{3\rho_{\Omega^*}(\zeta)}{\sqrt{\pi}}
\left(\int_\Omega \frac{|\mu(f_\nu^{-1}(w))-\nu(f_\nu^{-1}(w))|^2}{(1-|\mu(f_\nu^{-1}(w))|^2)(1-|\nu(f_\nu^{-1}(w))|^2)}
\frac{dudv}{|w-\zeta|^4}\right)^{1/2}
\end{align*}
holds for $\zeta \in \Omega^*$, where $\rho_{\Omega^*}(\zeta)$ is the hyperbolic density on $\Omega^*$.
\end{proposition}

\begin{proof}
It was shown in Yanagishita \cite[Lemma 3.1]{Yan} applying the argument of
Astala and Zinsmeister \cite{AZ} that 
$$
|S_{f_\mu \circ f_\nu^{-1}|_{\Omega^*}}(\zeta)|=
\frac{3}{2\pi}\rho_{\Omega^*}^2(\zeta)
\left|\int_{\lambda^{-1}(\Omega)}\bar \partial_{z} G(z,\zeta)dxdy \right|,
$$
where $G(\cdot,\zeta)=\kappa(\cdot,\zeta) \circ f_\mu \circ f_\nu^{-1} \circ \lambda(\cdot,\zeta)$ and
$$
\kappa(z,\zeta)=-\frac{(|f_\nu^{-1}(\zeta)|^2-1)f'_\mu(f_\nu^{-1}(\zeta))}{z-f_\mu(f_\nu^{-1}(\zeta))},\quad
\lambda(z,\zeta)=\zeta-\frac{(|f_\nu^{-1}(\zeta)|^2-1)f'_\nu(f_\nu^{-1}(\zeta))}{z}.
$$
Then, by the estimate using the Cauchy-Schwarz inequality and
the area theorem as in the proof of \cite[Proposition 3.2]{Yan}, we have that
\begin{align*}
\left|\int_{\lambda^{-1}(\Omega)}\bar \partial_{z} G(z,\zeta)dxdy \right|^2
&\leq 4\pi \rho_{\Omega^*}^{-2}(\zeta) \int_{\Omega}\frac{|\alpha(w)|^2}
{1-|\alpha(w)|^2}\frac{dudv}{|w-\zeta|^4},
\end{align*}
where $\alpha$ is the complex dilatation of $f_\mu \circ f_\nu^{-1}$.
This yields the required inequality.
\end{proof}

\begin{proposition}\label{base2}
For $\mu_1, \mu_2, \nu \in \Bel(\D)$, we have
$$
|r_\nu(\mu_1)(\zeta)-r_\nu(\mu_2)(\zeta)| \leq \frac{|\mu_1(z)-\mu_2(z)|}{\sqrt{(1-|\mu_1(z)|^2)(1-|\mu_2(z)|^2)}}
$$
for $\zeta=f^\nu(z)$ with $z \in \D$.
\end{proposition}

\begin{proof}
A simple computation shows that
\begin{align*}
|r_\nu(\mu_1)(\zeta)-r_\nu(\mu_2)(\zeta)|&=|\mu_1 \ast \nu^{-1}(\zeta)-\mu_2 \ast \nu^{-1}(\zeta)|\\
&= \left|\frac{\mu_1(z)-\nu(z)}{1-\overline{\nu(z)}\mu_1(z)}-\frac{\mu_2(z)-\nu(z)}{1-\overline{\nu(z)}\mu_2(z)}\right|\\
&= \frac{|\mu_1(z)-\mu_2(z)|(1-|\nu(z)|^2)}{|1-\overline{\nu(z)}\mu_1(z)||1-\overline{\nu(z)}\mu_2(z)|}
\leq \frac{|\mu_1(z)-\mu_2(z)|}{\sqrt{(1-|\mu_1(z)|^2)(1-|\mu_2(z)|^2)}}
\end{align*}
for $\zeta=f^\nu(z)$.
\end{proof}

\section{The $p$-integrable Teich\-m\"ul\-ler space}\label{4}

In this section, we prove the affine foliated structure of the
universal Teich\-m\"ul\-ler space $T$ 
induced by the $p$-integrable Teich\-m\"ul\-ler space $T^p$ for $p \geq 2$.
Later in this section, this is also extended to the Teich\-m\"ul\-ler space $T(\D/\Gamma)$ and 
the $p$-integrable Teich\-m\"ul\-ler space $T^p(\D/\Gamma)$ of a Riemann surface $\D/\Gamma$
for a certain Fuchsian group $\Gamma$.

A Beltrami coefficient $\mu \in \Bel(\D)$ is {\it $p$-integrable} for $p \geq 1$ if
$$
\Vert \mu \Vert_p^p:=\int_{\D} |\mu(z)|^p \rho_{\D}^2(z)dxdy <\infty,
$$
where $\rho_{\D}(z)=2/(1-|z|^2)$ is the hyperbolic density on $\D$.
The space of all $p$-integrable Beltrami coefficients on $\D$ is denoted by $\Ael^p(\D)$.
The $p$-integrable Teichm\"ulcer spaces defined below have been studied by Cui \cite{Cui}, Guo \cite{Guo}, Shen \cite{Sh},
Takhtajan and Teo \cite{TT},
Tang \cite{Tang} and Yanagishita \cite{Yan} among others.

\begin{definition}
A quasisymmetric homeomorphism $g:\S1 \to \S1$ belongs to $\Sym^p$ for $p \geq 2$ if
$g$ has a quasiconformal extension $\widetilde g: \D \to \D$ whose complex dilatation $\mu_{\widetilde g}$
belongs to $\Ael^p(\D)$. The {\it $p$-integrable Teichm\"ulcer space} $T^p$ is defined by
$$
T^p=\pi(\Ael^p(\D))=\Mob(\S1) \backslash \Sym^p \subset T.
$$
The topology on $T^p$ is induced by a norm $\Vert \cdot \Vert_p+\Vert \cdot \Vert_\infty$ on $\Ael^p(\D)$.
\end{definition}

We also consider the space of all $p$-integrable holomorphic quadratic automorphic forms on $\D^*$:
$$
A^p(\D^*)=\{\varphi \in B(\D^*) \mid \Vert \varphi \Vert_p^p:=\int_{\mathbb D^*} 
\rho^{2-2p}_{\D^*}(z)|\varphi(z)|^p dxdy <\infty \}.
$$
It is known that 
$A^p(\D^*) \subset B_0(\D^*)$.
It was proved in \cite[Theorem 2]{Cui} and \cite[Theorem 2]{Guo} that the Bers embedding $\beta$ of $T^p$ 
is a homeomorphism onto the image and satisfies 
$$
\beta(T^p)=\beta(T) \cap A^p(\D^*)
$$
for $p \geq 2$. This in particular implies that $T^p \subset T_0$ and hence
$\Sym^p \subset \Sym$.

For the map $\sigma:\Bel(\D) \to \Bel(\D)$ given by the barycentric extension, we have that $\sigma(\Ael^p(\D)) \subset \Ael^p(\D)$
(see \cite[Theorem 6]{Cui}, \cite[Theorem 2.1]{Tang}, \cite[Theorem 2.4]{Yan}).
Moreover, if $\nu \in \sigma(\Ael^p(\D))$, then $r_\nu$ preserves $\Ael^p(\D)$ and gives a
biholomorphic automorphism of $\Ael^p(\D)$.
Therefore, for arbitrary $\mu$ and $\nu$ in $\Ael^p(\D)$, we have that $\pi(r_\nu(\mu)) \in T^p$.
If $\tau \in T^p$ then
the base point change $R_\tau$ preserves $T^p$ and gives a
biholomorphic automorphism of $T^p$ (see \cite[Theorem 4]{Cui}, \cite[Lemma 3.4]{TT}, \cite[Proposition 5.1]{Yan}).

We state the main result in this section.

\begin{theorem}\label{A^p_B}
For each $\nu \in \Bel(\D)$, let $\psi=\Phi(\nu) \in B(\D^*)$. Then,
$$
\Phi \circ r_\nu^{-1}(\Ael^p(\D))=\beta(T) \cap \{A^p(\D^*)+\psi\}.
$$
Hence, $\beta \circ R_{\tau}^{-1}(T^p)=\beta(T) \cap \{A^p(\D^*)+\beta(\tau)\}$ for every
$\tau \in T$.
\end{theorem}

\begin{proof}
The inclusion $\subset$ in the case of $p=2$
was shown by Takhtajan and Teo \cite[Theorem 2.13]{TT}.
For a general $p \geq 2$, we can prove this by using the following claims 
based on Propositions \ref{base1} and \ref{base2}. They were
shown in 
Lemma 6.4 and Proposition 7.1 of \cite{Mat7} and the remarks after them, respectively.
We note that the inequalities below include the case where the right-hand side is $\infty$.

\smallskip
\noindent
{\it Claim 1.}
Let $\mu \in \Bel(\D)$ be arbitrary and let $\mu' \in \sigma(\Bel(\D))$ 
be obtained by the barycentric extension.
Then,
$$
\Vert \Phi(\mu)-\Phi(\mu') \Vert_p \leq \frac{C_1 \Vert \mu-\mu' \Vert_p}{\sqrt{(1-\Vert \mu \Vert_\infty^2)(1-\Vert \mu' \Vert_\infty^2)}},
$$
where $C_1>0$ is a constant depending only on $\Vert \mu' \Vert_\infty$.

\medskip
\noindent
{\it Claim 2.}
For $\mu_1, \mu_2 \in \Bel(\D)$ and $\nu^{-1} \in \sigma(\Bel(\D))$, we have that 
$$
\Vert r_\nu(\mu_1)-r_\nu(\mu_2) \Vert_p \leq C_2 \Vert \mu_1-\mu_2 \Vert_{p},
$$
where $C_2>0$ is a constant depending only on $\Vert \nu \Vert_\infty$,
$\Vert \mu_1 \Vert_\infty$, and $\Vert \mu_2 \Vert_\infty$.
\smallskip

For any $\mu \in \Ael^p(\D)$ and $\nu \in \sigma(\Bel(\D))$, we have only to show that $\Phi(r_\nu^{-1}(\mu))-\Phi(\nu) \in A^p(\D^*)$.
By Claim 2 applied to the right translation $r_{\nu^{-1}}=r_{\nu}^{-1}$, we have that
$$
\Vert r_\nu^{-1}(\mu)-\nu \Vert_p=\Vert r_\nu^{-1}(\mu)-r_\nu^{-1}(0) \Vert_p \leq C_2 \Vert \mu \Vert_p<\infty.
$$ 
Then, Claim 1 yields that
$$
\Vert \Phi(r_\nu^{-1}(\mu))-\Phi(\nu) \Vert_p \leq \frac{C_1\Vert r_\nu^{-1}(\mu)-\nu \Vert_p}{\sqrt{(1-\Vert r_\nu^{-1}(\mu) \Vert_\infty^2)(1-\Vert \nu \Vert_\infty^2)}}<\infty.
$$
This proves the inclusion $\subset$.

For the other inclusion $\supset$,
we take $\varphi \in A^p(\D^*)$ such that $\varphi +\psi \in \beta(T)$. Since $A^p(\D^*) \subset B_0(\D^*)$,
Proposition \ref{basic} asserts that
there is
a quasiconformal homeomorphism $\widehat f:\Chat \to \Chat$ conformal on $\Omega^*$
and asymptotically conformal on $\Omega$ such that $S_{\widehat f \circ f_\nu|_{\D^*}}=\varphi +\psi$.
According to Proposition \ref{summary}, we consider the decomposition $\widehat f=\widehat f_0 \circ \widehat f_1$
together with other maps that appear in this proposition. We use the properties shown in this proposition.

Since $\varphi \in A^p(\D^*)$, if $\varphi-\varphi_1 \in A^p(\D^*)$, then $\varphi_1 \in A^p(\D^*)$.
By property (2), $\mu_0$ in particular belongs to $\Ael^p(\D)$, and then by property (3) and 
the consequence from Claims 1 and 2 just proved above, we have that $\varphi-\varphi_1 \in A^p(\D^*)$.
Hence, $\varphi_1 \in A^p(\D^*)$.

We estimate the $p$-norm of $\widehat \mu_1 \circ f_\nu$. Property (1) yields that
$$
\int_{\D} |\widehat \mu_1 \circ f_\nu(z)|^p \rho_{\D}^2(z) dxdy
\leq \frac{1}{\delta} \int_{\D} (\rho_{\D^*}^{-2}(z^*)|\varphi_1(z^*)|)^p \rho^2_{\D}(z)dxdy.
$$
Here, we change the variables from $z=x+iy$ to $z^*=(\bar z)^{-1}=x^*+iy^*$. 
By $\rho_{\D}^2(z)dxdy=\rho_{\D^*}^2(z^*)dx^* dy^*$, the last integral is equal to
$$
\int_{\D^*} |\varphi_1(z^*)|^p \rho^{2-2p}_{\D^*}(z^*)dx^*dy^*,
$$
which is finite by $\varphi_1 \in A^p(\D^*)$. Hence, $\widehat \mu_1 \circ f_\nu \in \Ael^p(\D)$.

We will show that the complex dilatation $\mu_1$ of $f_1:\D \to \D$
belongs to $\Ael^p(\D)$. Since $|\widehat \mu_1 \circ f_\nu|=|\mu_1 \circ f^\nu|$,
we have that $\mu_1 \circ f^\nu \in \Ael^p(\D)$.
Then,
$$
\int_{\D}|\mu_1(\zeta)|^p \rho_{\D}^2(\zeta) d\xi d\eta
=\int_{\D}|\mu_1(f^\nu(z))|^p \rho_{\D}^2(f^\nu(z)) J_{f^\nu}(z)dx dy
$$
for $\zeta=f^\nu(z)$.
We may assume that $f^\nu$ is obtained by the barycentric extension. In this case,
the Jacobian $J_{f^\nu}$ is estimated as
$$
\rho_{\D}^2(f^\nu(z)) J_{f^\nu}(z) \leq C \rho_{\D}^2(z),
$$
where $C>0$ is a constant depending only on $\Vert \nu \Vert_\infty$ (\cite[Theorem 2]{DE}).
Thus, we see that the above integral is finite.

By property (2), the support of the complex dilatation $\mu_0$ of $f_0:\D \to \D$
is contained in a compact subset of $\D$.
Hence, we see that the complex dilatation $\mu_f=\mu_0 \ast \mu_1$ of $f=f_0 \circ f_1$
belongs to $\Ael^p(\D)$.
Since the complex dilatation on $\D$ of the quasiconformal homeomorphism $\widehat f \circ f_\nu$  
is $r_\nu^{-1}(\mu_f)$, we have that 
$$
\varphi +\psi=\Phi(r_\nu^{-1}(\mu_f)) \in \Phi \circ r_\nu^{-1}(\Ael^p(\D)),
$$
which shows the inclusion $\supset$.
\end{proof}

The above result can be generalized to the $p$-integrable Teich\-m\"ul\-ler space $T^p(\D/\Gamma)$ of
a Riemann surface $\D/\Gamma$
for a Fuchsian group $\Gamma$ 
whose hyperbolic elements have translation lengths uniformly bounded away from $0$.
We say that such a Fuchsian group $\Gamma$ satisfies the {\it Lehner condition}.
This Teich\-m\"ul\-ler space $T^p(\D/\Gamma)$ was introduced by Yanagishita \cite{Yan}. 

A Beltrami coefficient $\mu \in \Bel(\D)$ is $\Gamma$-invariant if
$$
(\gamma^* \mu)(z):=\mu(\gamma(z))\frac{\overline{\gamma'(z)}}{\gamma'(z)}=\mu(z) \quad ({\rm a.e.}\ z \in \D)
$$
for every $\gamma \in \Gamma$. Such a $\mu$ projects down to 
a Beltrami differential on the Riemann surface $\D/\Gamma$, and the space of all 
$\Gamma$-invariant Beltrami coefficients $\mu \in \Bel(\D)$ is
denoted by $\Bel(\D/\Gamma)$.
We define the space of all $\Gamma$-invariant $p$-integrable Beltrami coefficients on $\D$ as
$$
\Ael^p(\D/\Gamma)=\{\mu \in \Bel(\D/\Gamma) \mid 
\Vert \mu \Vert_p^p:=\int_{\D/\Gamma} |\mu(z)|^p \rho^2_{\D}(z)dxdy<\infty\}.
$$

A holomorphic quadratic automorphic form $\varphi \in B(\D^*)$ is $\Gamma$-invariant if
$$
(\gamma^* \varphi)(z):=\varphi(\gamma(z))\gamma'(z)^2=\varphi(z) \quad (\forall z \in \D)
$$
for every $\gamma \in \Gamma$. Such a $\varphi$ projects down to 
a holomorphic quadratic differentials on the Riemann surface $\D^*\!/\Gamma$, and the Banach space of all 
$\Gamma$-invariant bounded holomorphic quadratic automorphic forms $\varphi \in B(\D^*)$ is
denoted by $B(\D^*\!/\Gamma)$.
We define the Banach space of all $\Gamma$-invariant $p$-integrable holomorphic quadratic automorphic forms on $\D^*$ as
$$
A^p(\D^*\!/\Gamma)=\{\varphi \in B(\D^*\!/\Gamma) \mid 
\Vert \varphi \Vert_p^p:=\int_{\D^*/\Gamma} |\varphi(z)|^p \rho^{2-2p}_{\D^*}(z)dxdy<\infty\}.
$$

We apply the Teich\-m\"ul\-ler projection $\pi:\Bel(\D) \to T$ and 
the Bers projection $\Phi:\Bel(\D) \to B(\D^*)$ to $\Bel(\D/\Gamma)$ and $\Ael^p(\D/\Gamma)$, respectively.
The Teich\-m\"ul\-ler space $T(\D/\Gamma)$ of the Riemann surface $\D/\Gamma$ can be defined as
$T(\D/\Gamma)=\pi(\Bel(\D/\Gamma))$. The topology on $T(\D/\Gamma)$ is induced by the norm $\Vert \cdot \Vert_\infty$ on $\Bel(\D/\Gamma)$.
It is well-known (see \cite{Leh}) that the Bers embedding
$\beta$ restricted to $T(\D/\Gamma)$ is a homeomorphism onto the image
$$
\beta(T(\D/\Gamma))=\Phi(\Bel(\D/\Gamma))=\beta(T) \cap B(\D^*\!/\Gamma),
$$ 
which is a bounded domain in $B(\D^*\!/\Gamma)$.
By the conformal naturality of the barycentric extension, we have $\sigma(\Bel(\D/\Gamma)) \subset \Bel(\D/\Gamma)$.

The $p$-integrable Teich\-m\"ul\-ler space $T^p(\D/\Gamma)$ is defined by
$T^p(\D/\Gamma)=\pi(\Ael^p(\D/\Gamma))$. 
The topology on $T^p(\D/\Gamma)$ is induced by the norm $\Vert \cdot \Vert_p+\Vert \cdot \Vert_\infty$ on $\Ael^p(\D/\Gamma)$.
It was shown in \cite{Yan} that
$\beta$ restricted to $T^p(\D/\Gamma)$ is a homeomorphism onto the image
$$
\beta(T^p(\D/\Gamma))=\Phi(\Ael^p(\D/\Gamma))=\beta(T) \cap A^p(\D^*\!/\Gamma),
$$
which is a domain in $A^p(\D^*\!/\Gamma)$.
The properties of $\sigma:\Ael^p(\D/\Gamma) \to \Ael^p(\D/\Gamma)$ induced by the barycentric extension,
the right translation $r_\nu:\Ael^p(\D/\Gamma) \to \Ael^p(\D/\Gamma)$, and the base point change 
$R_\tau:T^p(\D/\Gamma) \to T^p(\D/\Gamma)$ is the same as
in the case where $\Gamma$ is trivial, which are all given in \cite{Yan}.

The $\Gamma$-invariant version of the previous theorem is as follows.

\begin{theorem}\label{A^p_B-group}
We assume that a Fuchsian group $\Gamma$ satisfies the Lehner condition.
For each $\nu \in \Bel(\D/\Gamma)$, let $\psi=\Phi(\nu) \in B(\D^*\!/\Gamma)$. Then,
$$
\Phi \circ r_\nu^{-1}(\Ael^p(\D/\Gamma))=\beta(T(\D/\Gamma)) \cap \{A^p(\D^*\!/\Gamma)+\psi\}.
$$
Hence, $\beta \circ R_{\tau}^{-1}(T^p(\D/\Gamma))=\beta(T(\D/\Gamma)) \cap \{A^p(\D^*\!/\Gamma)+\beta(\tau)\}$ for every
$\tau \in T(\D/\Gamma)$.
\end{theorem}

\begin{proof}
We only refer to the points where the compatibility with $\Gamma$ is required in addition to the proof of Theorem \ref{A^p_B}.
The other places are similarly carried out by the previous arguments.
(1) The corresponding statement to Theorem \ref{B_0-B} is valid for $T(\D/\Gamma)$ and $T_0(\D/\Gamma)$
by Earle, Markovic and Saric \cite[Theorem 4]{EMS}.
(2) Lemma \ref{reflection} and Proposition \ref{summary} respecting the compatibility with $\Gamma$ are already valid 
as they are since the Earle--Nag quasiconformal reflection is used for constructing a local section
of the Bers projection (see \cite[Section 14.3]{GL}).
(3) The corresponding statement to Claim 1 can be obtained by also using the proof of \cite[Proposition 3.2]{Yan}.
(4) The corresponding statement to Claim 2 can be obtained by also using the proof of \cite[Proposition 5.1]{Yan}.
\end{proof}

Theorem \ref{A^p_B-group} in particular implies that $\beta(T(\D/\Gamma)) \cap \{A^p(\D^*\!/\Gamma)+\beta(\tau)\}$
is connected for each $\tau \in T(\D/\Gamma)$, and any such components are biholomorphically  equivalent to
each other, which admits the Banach manifold structure.

We consider the quotient spaces with the projections $T(\D/\Gamma) \to T^p(\D/\Gamma) \backslash T(\D/\Gamma)$ and
$B(\D^*\!/\Gamma) \to A^p(\D^*\!/\Gamma) \backslash B(\D^*\!/\Gamma)$. These maps are continuous and open
with respect to the quotient topology. Then, we see that the projection of the Bers embedding $\beta$ is not only
well-defined to be an injection but also a homeomorphism onto the image.

\begin{corollary}
The quotient Bers embedding 
$$
\widehat \beta^p:T^p(\D/\Gamma) \backslash T(\D/\Gamma) \to A^p(\D^*\!/\Gamma) \backslash B(\D^*\!/\Gamma)
$$ 
is well-defined to be a homeomorphism onto the image
for any Fuchsian group $\Gamma$ satisfying the Lehner condition.
\end{corollary}

\section{The Teich\-m\"ul\-ler space of circle diffeomorphisms of H\"older continuous derivative}\label{5}

In this section, we prove the affine foliated structure of the
universal Teich\-m\"ul\-ler space $T$ 
induced by the Teich\-m\"ul\-ler space $T_0^{>0}$ of circle diffeomorphisms of H\"older continuous derivative.

For a constant $\alpha \in (0,1)$,
we denote by $\Diff_+^{1+\alpha}(\S1)$ the group of all orientation-preserving diffeomorphisms $g$ of 
the unit circle $\S1$ whose derivatives are $\alpha$-H\"older continuous;
there is a constant $c \geq 0$ such that
$$
|g'(x)-g'(y)| \leq c|x-y|^\alpha
$$
for any $x,y \in \S1=\R/\mathbb Z$.
We give a characterization of $\Diff^{1+\alpha}(\S1)$ 
analogously to Proposition \ref{gs} by considering the following spaces:
\begin{align*}
& T^\alpha_0=\Mob(\S1) \backslash \Diff^{1+\alpha}_+(\S1);\\
& \Bel^{\alpha}_0(\D)=\{\mu \in \Bel_0(\D) \mid \Vert \mu \Vert_{\infty,\alpha}:=
\esssup_{z \in \D} \rho^{\alpha}_{\D}(z)|\mu(z)|<\infty\};\\
& B^{{ \alpha}}_0(\D^*)
=\{\varphi \in B_0(\D^*) \mid \Vert \varphi \Vert_{\infty,\alpha}:=\sup_{z \in \mathbb D^*} 
\rho^{-2+{ \alpha}}_{\D^*}(z)|\varphi(z)|<\infty \}.
\end{align*}
We define $T^\alpha_0$ to be the {\it Teich\-m\"ul\-ler space of circle diffeomorphisms} 
of $\alpha$-H\"older continuous derivatives (\cite{Mat0, Mat2, Mat3, Mat4}).

\begin{proposition}\label{main}
For a quasisymmetric homeomorphism $g \in \QS$, the following conditions are equivalent:
$(1)$ $g \in \Diff^{1+\alpha}_+(\S1)$; $(2)$ $s([g]) \in \Bel_0^\alpha(\D)$; $(3)$ $\beta([g]) \in B_0^\alpha(\D^*)$.
\end{proposition}

We consider 
the Teich\-m\"ul\-ler space $T^{>\alpha}_0=\bigcup_{\varepsilon >0} T^{\alpha+\varepsilon}_0$ 
of circle diffeomorphisms of H\"older continuous derivative
with exponent greater than $\alpha \in[0,1)$. In other words, 
for $\Diff_+^{>\alpha}(\S1)=\bigcup_{\varepsilon >0} \Diff_+^{\alpha+\varepsilon}(\S1)$,
we define $T^{>\alpha}_0=\Mob(\S1) \backslash \Diff_+^{>\alpha}(\S1)$.
Correspondingly,
we set the unions of increasing subspaces as
$$
\Bel_0^{>\alpha}(\D)=\bigcup_{\varepsilon>0}\Bel_0^{\alpha+\varepsilon}(\D);\quad
B^{>\alpha}_0(\D^*)=\bigcup_{\varepsilon>0}B^{\alpha+\varepsilon}_0(\D^*).
$$
The norms on these spaces can be given by $\lim_{\varepsilon \to 0} \Vert \mu \Vert_{\infty,\alpha+\varepsilon}$ and
$\lim_{\varepsilon \to 0} \Vert \varphi \Vert_{\infty,\alpha+\varepsilon}$, respectively.
We see that $T_0^{>\alpha}=\pi(\Bel_0^{>\alpha}(\D))$ and $\beta(T_0^{>\alpha})=\beta(T) \cap B^{>\alpha}_0(\D^*)$. In particular,
for $\alpha=0$, we have the Teich\-m\"ul\-ler space 
$T_0^{>0}$ of circle diffeomorphisms 
of H\"older continuous derivative with an arbitrary exponent.

The following lemma plays the role of the combination of Claims 1 and 2 in the previous proof,
and used to show that the quotient Bers embedding is well-defined. 
This lemma also tells us that the exact $T_0^\alpha$ does not seem to induce an affine foliated structure 
for $T$.

\begin{lemma}\label{base3}
Let $\nu \in \Bel(\D)$ hold $\Vert \nu \Vert_\infty < k<1$ and put $K=(1+k)/(1-k)$. Then,
$$
\Vert \Phi(r_\nu(\mu_1))-\Phi(r_\nu(\mu_2)) \Vert_{\infty,\,\alpha/K^2} \leq C \Vert \mu_1-\mu_2 \Vert_{\infty,\alpha}
$$
is satisfied 
for any $\mu_1, \mu_2 \in \Bel_0^\alpha(\D)$ with $\alpha \in (0,1)$, where $C>0$ depends only on 
$k$, $\Vert \mu_1 \Vert_\infty$, and $\Vert \mu_2 \Vert_{\infty,\alpha}$.
\end{lemma}

\begin{proof}
The Mori theorem (\cite[Section III.C]{Ah0}) implies that there is some constant $C_0 \geq 1$ 
depending only on $k$ such that
$$
\frac{1}{C_0}(1-|z|)^K \leq 1-|\zeta| 
$$
for $\zeta=f^\nu(z)$.
Then, 
$$
\rho_{\D}^{\alpha/K}(\zeta) =\left(\frac{2}{1-|\zeta|^2}\right)^{\alpha/K}
\leq 2 C_0 \left(\frac{2}{1-|z|^2}\right)^{\alpha}
=2 C_0 \rho_{\D}^{\alpha}(z),
$$
and it follows from Proposition \ref{base2} that
$$
\rho_{\D}^{\alpha/K}(\zeta)|r_\nu(\mu_1)(\zeta)-r_\nu(\mu_2)(\zeta)| \leq C_1 \rho_{\D}^{\alpha}(z)|\mu_1(z)-\mu_2(z)|
$$
for $C_1=2C_0/\sqrt{(1-\Vert \mu_1 \Vert^2_\infty)(1-\Vert \mu_2 \Vert^2_\infty)}$. 
Thus, we have
$$
\Vert r_\nu(\mu_1)-r_\nu(\mu_2) \Vert_{\infty,\alpha/K} \leq C_1 \Vert \mu_1-\mu_2 \Vert_{\infty,\alpha}.
$$

We see from the proof of \cite[Lemma 3.3]{Mat7} based on Proposition \ref{base1} that
$$
\Vert \Phi(\mu'_1)-\Phi(\mu'_2) \Vert_{\infty,\,\alpha/\widetilde K} 
\leq C_2 \Vert \mu'_1-\mu'_2 \Vert_{\infty,\alpha}
$$
for $\widetilde K=(1+\tilde k)/(1-\tilde k)$ with $\tilde k=\Vert \mu'_2 \Vert_\infty$,
where $C_2>0$ is a constant given by $\Vert \mu'_1 \Vert_\infty$ and $\Vert \mu'_2 \Vert_\infty$. 
We apply this inequality for $\mu'_1=r_\nu(\mu_1)$ and $\mu'_2=r_\nu(\mu_2)$ to obtain
$$
\Vert \Phi(r_\nu(\mu_1))-\Phi(r_\nu(\mu_2)) \Vert_{\infty,\,\alpha/(K \widetilde K)} 
\leq C_2 \Vert r_\nu(\mu_1)-r_\nu(\mu_2) \Vert_{\infty,\alpha/K}.
$$
However, since $\mu_2 \in \Bel_0^\alpha(\D)$, 
we can assume 
$\tilde k=\Vert r_\nu(\mu_2) \Vert_\infty$ to be arbitrarily close to $\Vert \nu \Vert_\infty$
by allowing $C_2$ to depend on $\Vert \mu_2 \Vert_{\infty,\alpha}$.
Indeed, to see this, we use the decomposition $\mu_2=\mu_{20} \ast \mu_{21}$ such that $\mu_{20}$ has a compact support in $\D$ and
$\Vert \mu_{21} \Vert_\infty$ is arbitrarily small. 
Then, we apply the previous estimate for
$\Vert r_\nu(\mu_1)-r_\nu(\mu_2) \Vert_{\infty,\alpha/K}$ and conclude that
$$
\Vert \Phi(r_\nu(\mu_1))-\Phi(r_\nu(\mu_2)) \Vert_{\infty,\,\alpha/K^2} 
\leq C_1 C_2 \Vert \mu_1-\mu_2 \Vert_{\infty,\alpha}.
$$
Here, the constant $C$ 
can be chosen as $C=C_1 C_2$ which depends only on 
$k$, $\Vert \mu_1 \Vert_\infty$, and $\Vert \mu_2 \Vert_{\infty,\alpha}$.
\end{proof}

We state the main result in this section. 
The arguments are parallel to those for Theorem \ref{A^p_B}
in some parts.

\begin{theorem}\label{compatible3}
For each $\nu \in \Bel(\D)$, let $\psi=\Phi(\nu) \in B(\D^*)$. Then,
$$
\Phi \circ r_\nu^{-1}(\Bel_0^{>0}(\D)) = \beta(T) \cap \{B_0^{>0}(\D^*)+\psi\}.
$$
Hence, $\beta \circ R_{\tau}^{-1}(T_0^{>0}) = \beta(T) \cap \{B_0^{>0}(\D^*)+\beta(\tau)\}$ for every
$\tau \in T$.
\end{theorem}

\begin{proof}
For one inclusion $\subset$, we
take an arbitrary $\mu \in \Bel_0^{>0}(\D)$. There is some $\alpha \in (0,1)$ such that 
$\mu \in \Bel_0^{\alpha}(\D)$. 
We apply Lemma \ref{base3} for $\mu_1=\mu$ and $\mu_2=0$ with the right translation $r_{\nu^{-1}}=r_\nu^{-1}$.
Then, we obtain
$$
\Vert \Phi(r^{-1}_\nu(\mu))-\Phi(\nu) \Vert_{\infty,\,\alpha/K^2}
= \Vert \Phi(r^{-1}_\nu(\mu))-\Phi(r^{-1}_\nu(0)) \Vert_{\infty,\,\alpha/K^2}
\leq C \Vert \mu \Vert_{\infty,\alpha}<\infty.
$$
This implies that
$\Phi (r_\nu^{-1}(\mu)) \in B_0^{>0}(\D^*)+\Phi(\nu)$, and hence 
$\Phi \circ r_\nu^{-1}(\Bel_0^{>0}(\D))$ is contained in $B_0^{>0}(\D^*)+\psi$.

For the other inclusion $\supset$,
we take $\varphi \in B_0^{>0}(\D^*)$ such that $\varphi+\psi \in \beta(T)$. Since 
$B_0^{>0}(\D^*) \subset B_0(\D^*)$, Proposition \ref{basic} asserts that
there is
a quasiconformal homeomorphism $\widehat f:\Chat \to \Chat$ conformal on $\Omega^*$
and asymptotically conformal on $\Omega$ such that $S_{\widehat f \circ f_\nu|_{\D^*}}=\varphi+\psi$.
According to Proposition \ref{summary}, we consider the decomposition $\widehat f=\widehat f_0 \circ \widehat f_1$
with related maps.
Since $\varphi \in B_0^{>0}(\D^*)$, if $\varphi-\varphi_1 \in B_0^{>0}(\D^*)$, then $\varphi_1 \in B_0^{>0}(\D^*)$.
By property (2), $\mu_0$ in particular belongs to $\Bel_0^{>0}(\D)$, and then by property (3) and 
the consequence from Lemma \ref{base3} just proved above, we have that $\varphi-\varphi_1 \in B_0^{>0}(\D^*)$.
Hence, $\varphi_1 \in B_0^{>0}(\D^*)$. 
Property (1) yields that
$$
|\widehat \mu_1 \circ f_\nu(z)|
\leq \frac{1}{\delta}\, \rho_{\D^*}^{-2}(z^*)|\varphi_1(z^*)|.
$$
This implies that $\widehat \mu_1 \circ f_\nu \in \Bel_0^{>0}(\D)$.

We will show that the complex dilatation $\mu_1$ of $f_1:\D \to \D$
belongs to $\Bel_0^{>0}(\D)$. Since $|\widehat \mu_1 \circ f_\nu|=|\mu_1 \circ f^\nu|$,
we have that $\mu_1 \circ f^\nu \in \Bel_0^{>0}(\D)$. By the Mori theorem at the boundary, we know that
$$
\frac{1}{C_0} (1-|z|)^{-1/K} \leq 1-|f^\nu(z)| 
\quad (z \in \D)
$$
for some constant $C_0 \geq 1$, where $K=(1+\Vert \nu \Vert_\infty)/(1-\Vert \nu \Vert_\infty)$.
From this estimate, we see that $\mu_1 \in \Bel_0^{>0}(\D)$.

By property (2), the support of $\mu_0$ 
is contained in a compact subset of $\D$.
Hence, we see that the complex dilatation $\mu_f=\mu_0 \ast \mu_1$ of $f=f_0 \circ f_1$
belongs to $\Bel_0^{>0}(\D)$.
Since the complex dilatation on $\D$ of the quasiconformal homeomorphism $\widehat f \circ f_\nu$  
is $r_\nu^{-1}(\mu_f)$, we have that 
$$
\varphi +\psi=\Phi(r_\nu^{-1}(\mu_f)) \in \Phi \circ r_\nu^{-1}(\Bel_0^{>0}(\D)),
$$
which shows the inclusion $\supset$.
\end{proof}

\begin{corollary}
The quotient Bers embedding 
$\widehat \beta^{>0}_0:T_0^{>0} \backslash T \to B_0^{>0}(\D^*) \backslash B(\D^*)$ is well-defined to be
a homeomorphism onto the image.
\end{corollary}

\section{Stratification of foliated structures}\label{6} 

We can also consider an affine foliated structure for a certain subspace of the universal Teich\-m\"ul\-ler space $T$
induced by a smaller subspace. We consider again an affine foliated structure for this smaller subspace, and repeat this
process.
Then, we obtain a stratification of affine foliated structures of $T$.
In this section, we observe such an example.

First, we show the affine foliated structure for the little Teich\-m\"ul\-ler space $T_0$ induced by
$T_0^{>\alpha}$, the Teich\-m\"ul\-ler space of circle diffeomorphisms of H\"older continuous derivative
of exponent greater than $\alpha \in [0,1)$.
We in particular have the affine foliated structure  
for the little subspace $T_0$ by $T^{>0}_0$. 

We prepare the asymptotically conformal version of 
the Mori theorem at the boundary.
The corresponding result under a stronger assumption that the complex dilatation $\mu$ has an explicit decay order
as in $\Bel^\alpha_0(\D)$ is given in \cite[Theorem 6.4] {Mat2}.

\begin{lemma}\label{mori2}
Let $f^\nu \in \AC(\D)$ be a normalized asymptotically conformal homeomorphism of $\D$ with the complex dilatation $\nu\in \Bel_0(\D)$. 
Let $\varepsilon >0$ be an arbitrary positive constant.
Then, there is a constant $A \geq 1$ depending only on 
$\nu$ and $\varepsilon$ such that
$$
\frac{1}{A} (1-|z|)^{1+\varepsilon} \leq 1-|f^\nu(z)| \leq A(1-|z|)^{1-\varepsilon}
$$
for every $z \in \D$.
\end{lemma}

\begin{proof}
Since $\nu \in \Bel_0(\D)$,
we can find $t_0 \in (0,1/4)$ so that $|\nu(\zeta)| \leq \varepsilon/2$ for almost every $\zeta \in \D$ with $|\zeta|>1-\sqrt{t_0}$.
This depends on $\nu$ and $\varepsilon$.
We define a Beltrami coefficient $\nu_0(\zeta)$ by
setting $\nu_0(\zeta) = \nu(\zeta)$ on $|\zeta| \leq \sqrt{t_0}$ and $\nu_0(\zeta) \equiv 0$ elsewhere.
Let $f_0$ be the normalized quasiconformal homeomorphism of $\D$ with the complex dilatation $\nu_0$.
Let $f_1$ be the quasiconformal homeomorphism of $\D$ such that $f = f_1 \circ f_0$. 
For $K=(1+\varepsilon/2)/(1-\varepsilon/2)$, we see that $f_1$ is a $K$-quasiconformal homeomorphism of $\D$. Here, we have
$$
\frac{1}{K}=\frac{1-\varepsilon/2}{1+\varepsilon/2} \geq 1-\varepsilon.
$$

First, we apply a distortion theorem to the conformal homeomorphism $f_0(\zeta)$ restricted to 
$\zeta \in \D$ with $|\zeta| >\sqrt{t_0}$.
In fact, we may assume that $f_0$ is a conformal homeomorphism of an annulus $\{\sqrt{t_0}<|\zeta|<1/\sqrt{t_0}\}$
by the reflection principle.  
Since $\S1$ is compact, there is some constant $L \geq 1$ such that
the modulus of the derivative $|f_0'(\xi)|$ at any $\xi \in \S1$ is bounded by $L$,
which is depending only on $t_0$. 
The Koebe distortion theorem (see \cite[Theorem 1.3]{P}) in the disk of radius $\sqrt{t_0}$ and center 
$\xi=z/|z|$ yields
$$
1-|f_0(z)| \leq \frac{L(1-|z|)}{\{1-(1-|z|)/\sqrt{t_0}\}^2} \leq \frac{L(1-|z|)}{(1-\sqrt{t_0})^2} \leq 4L(1-|z|) 
$$
for every $z \in \D$ with $1-|z|<t_0$.

Next, we apply the Mori theorem (\cite[Section III.C]{Ah0}) to the quasiconformal homeomorphism $f_1$ of $\D$.
It implies that
$$
1-|f_1(w)| \leq 16(1-|w|)^{1/K} \leq 16(1-|w|)^{1-\varepsilon}
$$
for every $w \in \D$. Then, by setting $w=f_0(z)$, we have that
$$
1-|f(z)| \leq 16\{4L(1-|z|)\}^{1-\varepsilon} \leq 64 L(1-|z|)^{1-\varepsilon}.
$$

If $1-|z| \geq t_0$, we simply obtain
$$
1-|f(z)| \leq 1 \leq \frac{1}{t_0}(1-|z|)^{1-\varepsilon}. 
$$
Combined with the previous estimate, this gives the right side inequality in the statement.
For the left side inequality, we apply the lower estimates in both the Koebe and the Mori theorems, or
apply the above arguments to the inverse map $f^{-1}$.
\end{proof}

\begin{theorem}\label{compatible}
Let $\alpha \in [0,1)$ be an arbitrary exponent.
For each $\nu \in \Bel_0(\D)$, let $\psi=\Phi(\nu) \in B_0(\D^*)$. Then,
$$
\Phi \circ r_\nu^{-1}(\Bel_0^{>\alpha}(\D)) = \beta(T_0) \cap \{B_0^{>\alpha}(\D^*)+\psi\}.
$$
Hence, $\beta \circ R_{\tau}^{-1}(T_0^{>\alpha}) = \beta(T) \cap \{B_0^{>\alpha}(\D^*)+\beta(\tau)\}$ for every
$\tau \in T_0$.
\end{theorem}

\begin{proof}
For one inclusion $\subset$, we
take an arbitrary $\mu \in \Bel_0^{>\alpha}(\D)$. Then, there is some $\alpha' \in (0,1)$ such that $\alpha<\alpha'$ and
$\mu \in \Bel_0^{\alpha'}(\D)$. We choose $\varepsilon>0$ such that $\alpha+\varepsilon<\alpha'$.
We apply \cite[Theorem 3.6]{Mat7} based on Propositions \ref{base1} and \ref{base2} 
to obtain that
$$
\Vert \Phi(r_\nu^{-1}(\mu))-\Phi(\nu) \Vert_{\infty,\alpha'-\varepsilon} 
=\Vert \Phi(r_\nu^{-1}(\mu))-\Phi(r_\nu^{-1}(0)) \Vert_{\infty,\alpha'-\varepsilon}
\leq C \Vert \mu \Vert_{\infty,\alpha'}<\infty
$$
for some constant $C>0$.
This implies that
$\Phi (r_\nu^{-1}(\mu))-\Phi(\nu) \in B_0^{>\alpha}(\D^*)$, and hence 
$\Phi \circ r_\nu^{-1}(\Bel_0^{>\alpha}(\D))$ is contained in $B_0^{>\alpha}(\D^*)+\psi$.

For the other inclusion $\supset$,
we take $\varphi \in B_0^{>\alpha}(\D^*)$ such that $\varphi+\psi \in \beta(T)$. 
By Proposition \ref{basic},
there is
a quasiconformal homeomorphism $\widehat f:\Chat \to \Chat$ conformal on $\Omega^*$
and asymptotically conformal on $\Omega$ such that $S_{\widehat f \circ f_\nu|_{\D^*}}=\varphi+\psi$.
We consider the decomposition $\widehat f=\widehat f_0 \circ \widehat f_1$ as in Proposition \ref{summary}.

By property (2), $\mu_0$ in particular belongs to $\Bel_0^{>\alpha}(\D)$, and then by property (3) and 
the result proved above, we have that $\varphi-\varphi_1 \in B_0^{>\alpha}(\D^*)$.
Hence, $\varphi_1 \in B_0^{>\alpha}(\D^*)$ follows from $\varphi \in B_0^{>\alpha}(\D^*)$.
Property (1) yields that
$$
|\widehat \mu_1 \circ f_\nu(z)|
\leq \frac{1}{\delta}\, \rho_{\D^*}^{-2}(z^*)|\varphi_1(z^*)|.
$$
This implies that $\widehat \mu_1 \circ f_\nu \in \Bel_0^{>\alpha}(\D)$.

Let $\mu_1$ be the complex dilatation of $f_1:\D \to \D$.
Since $|\widehat \mu_1 \circ f_\nu|=|\mu_1 \circ f^\nu|$,
we have that $\mu_1 \circ f^\nu \in \Bel_0^{\alpha'}(\D)$ for some $\alpha'>\alpha$. 
Having $\nu \in \Bel_0(\D)$, we apply Lemma \ref{mori2} 
to $\zeta=f^{\nu}(z)$. 
We choose $\varepsilon'>0$ such that $(1+\varepsilon')^{-1}\alpha'>\alpha$. Then, we have
$$
\frac{1}{A} (1-|z|)^{1+\varepsilon'} \leq 1-|f^{\nu}(z)| =1-|\zeta| 
$$
for some constant $A \geq 1$. This shows that $\mu_1 \in \Bel_0^{>\alpha}(\D)$. 

By property (2), we see that the complex dilatation $\mu_f=\mu_0 \ast \mu_1$ of $f=f_0 \circ f_1$
belongs to $\Bel_0^{>\alpha}(\D)$.
Since the complex dilatation on $\D$ of the quasiconformal homeomorphism $\widehat f \circ f_\nu$  
is $r_\nu^{-1}(\mu_f)$, we have that 
$$
\varphi +\psi=\Phi(r_\nu^{-1}(\mu_f)) \in \Phi \circ r_\nu^{-1}(\Bel_0^{>\alpha}(\D)),
$$
which shows the inclusion $\supset$.
\end{proof}

Next, we consider the affine foliated structure 
of the Teich\-m\"ul\-ler space $T^{>0}_0$ of circle diffeomorphisms of H\"older continuous derivative
induced by $T_0^\alpha$.

\begin{theorem}\label{compatible2}
For each $\nu \in \Bel_0^{>0}(\D)$, let $\psi=\Phi(\nu) \in B^{>0}_0(\D^*)$. Then,
$$
\Phi \circ r_\nu^{-1}(\Bel_0^{\alpha}(\D)) = \beta(T^{>0}_0) \cap \{B_0^{\alpha}(\D^*)+\psi\}.
$$
Hence, $\beta \circ R_{\tau}^{-1}(T_0^{\alpha}) = \beta(T^{>0}_0) \cap \{B_0^{\alpha}(\D^*)+\beta(\tau)\}$ for every
$\tau \in T^{>0}_0$.
\end{theorem} 

Theorem \ref{compatible2} in particular 
implies that the Teich\-m\"ul\-ler space $T^{>0}_0$ of circle diffeomorphisms of H\"older continuous derivatives
of any exponent has the complex Banach manifold structure modeled on $B_0^\alpha(\D^*)$ for every $\alpha \in (0,1)$ and each
connected component of $T^{>0}_0$ is biholomorphically equivalent to $T_0^\alpha$.
We note that this is not close in the topology of $T^{>0}_0$.
Moreover, the Bers embedding of $T^{>0}_0$ realizes the connected components of $T^{>0}_0$ as its affine foliation by $B_0^\alpha(\D^*)$.

\begin{corollary}\label{th4}
The Bers embedding $\beta:T_0^{>0} \to B_0^{>0}(\D^*)$ provides the complex Banach manifold structure for $T_0^{>0}$
modeled on $B_0^{\alpha}(\D^*)$ together with its affine foliation. 
\end{corollary}

\noindent
{\it Proof of Theorem \ref{compatible2}.}
For one inclusion $\subset$, we
take an arbitrary $\mu \in \Bel_0^{\alpha}(\D)$. 
By \cite[Theorem 3.6]{Mat7},
we obtain that
$$
\Vert \Phi(r_\nu^{-1}(\mu))-\Phi(\nu) \Vert_{\infty,\alpha} 
=\Vert \Phi(r_\nu^{-1}(\mu))-\Phi(r_\nu^{-1}(0)) \Vert_{\infty,\alpha} 
\leq C \Vert \mu \Vert_{\infty,\alpha}<\infty.
$$
for some constant $C>0$.
This implies that
$\Phi (r_\nu^{-1}(\mu)) -\Phi(\nu) \in B_0^{\alpha}(\D^*)$, and hence 
$\Phi \circ r_\nu^{-1}(\Bel_0^{\alpha}(\D))$ is contained in $B_0^{\alpha}(\D^*)+\psi$.

For the other inclusion $\supset$,
we take $\varphi \in B_0^{\alpha}(\D^*)$ such that $\varphi+\psi \in \beta(T)$. 
For a quasiconformal homeomorphism $\widehat f:\Chat \to \Chat$ conformal on $\Omega^*$
and asymptotically conformal on $\Omega$ such that $S_{\widehat f \circ f_\nu|_{\D^*}}=\varphi+\psi$ (Proposition \ref{basic}),
we consider the decomposition $\widehat f=\widehat f_0 \circ \widehat f_1$
as in Proposition \ref{summary}. 
By property (2), $\mu_0$ belongs to $\Bel_0^{\alpha}(\D)$, and then by property (3) with 
the above result, we have that $\varphi-\varphi_1 \in B_0^{\alpha}(\D^*)$.
Hence, $\varphi_1 \in B_0^{\alpha}(\D^*)$ follows from $\varphi \in B_0^{\alpha}(\D^*)$.
Property (1) yields that
$$
|\widehat \mu_1 \circ f_\nu(z)|
\leq \frac{1}{\delta}\, \rho_{\D^*}^{-2}(z^*)|\varphi_1(z^*)|,
$$
which implies that $\widehat \mu_1 \circ f_\nu \in \Bel_0^{\alpha}(\D)$.

We consider the complex dilatation $\mu_1$ of $f_1:\D \to \D$.
Since $|\widehat \mu_1 \circ f_\nu|=|\mu_1 \circ f^\nu|$,
we have that $\mu_1 \circ f^\nu \in \Bel_0^{\alpha}(\D)$. 
We apply the stronger version 
of the Mori theorem at the boundary (\cite[Theorem 6.4]{Mat2}) 
to $\zeta=f^{\nu}(z)$. Then
$$
\frac{1}{A} (1-|z|) \leq 1-|f^{\nu}(z)| =1-|\zeta| 
$$
for some constant $A \geq 1$. This shows that $\mu_1 \in \Bel_0^{\alpha}(\D)$. 
By property (3), the complex dilatation $\mu_f=\mu_0 \ast \mu_1$ of $f=f_0 \circ f_1$
belongs to $\Bel_0^{\alpha}(\D)$.
Since the complex dilatation of $\widehat f \circ f_\nu$ on $\D$ 
is $r_\nu^{-1}(\mu_f)$, we have that 
$$
\varphi +\psi=\Phi(r_\nu^{-1}(\mu_f)) \in \Phi \circ r_\nu^{-1}(\Bel_0^{\alpha}(\D)).
$$
This proves the assertion.
\qed
\medskip

Now we obtain a stratification of the affine foliated structures of the universal Teich\-m\"ul\-ler space $T$ as follows:
the first level is $T$ by $T_0$; the second level is $T_0$ by $T^{>0}_0$; and the third level 
is $T^{>0}_0$ by $T_0^\alpha$.

\section{Applications to representation spaces}\label{7}

The universal asymptotic Teich\-m\"ul\-ler space is given by $AT=\Sym \backslash \QS$
in Gardiner and Sullivan \cite{GS}.
This admits a complex structure modeled on the quotient Banach space $B_0(\D) \backslash B(\D)$.
See also Earle, Gardiner and Lakic \cite{EGL} for the asymptotic Teich\-m\"ul\-ler space $AT(\D/\Gamma)$ of a Riemann surface.
Similarly, for the group $\Diff_+^{>1}(\S1)$ of all circle diffeomorphisms of H\"older continuous derivatives, we can consider
the quotient space
$DT=\Diff_+^{>1}(\S1) \backslash \QS$. However, this is no more a Hausdorff space in the quotient topology of $\QS$.

We impose group compatibility on these spaces.
Let $\Gamma \subset \Mob(\S1) \cong {\rm PSL}(2,\mathbb R)$ be a non-elementary Fuchsian group. 
The deformation space of $\Gamma$ in $\Mob(\S1)$ is given as
the Teich\-m\"ul\-ler space of $\Gamma$, which is defined by
$$
T(\Gamma)=\Mob(\S1) \backslash \{f \in \QS \mid f\Gamma f^{-1} \subset \Mob(\S1)\} \subset T.
$$
This is a closed subspace of $T$, and can be identified with the Teich\-m\"ul\-ler space $T(\D/\Gamma)$ of
the Riemann surface $\D/\Gamma$.
In a similar manner, the deformation space of $\Gamma$ in $\Sym$ is given in \cite{Mat6} as 
the asymptotic Teich\-m\"ul\-ler space of $\Gamma$,
which is defined by
$$
AT(\Gamma)=\Sym \backslash \{f \in \QS \mid f\Gamma f^{-1} \subset { \Sym} \} \subset AT.
$$
We remark that this is different from the asymptotic Teich\-m\"ul\-ler space $AT(\D/\Gamma)$ of the Riemann surface $\D/\Gamma$
studied in \cite{EGL}.
We also define here the deformation space of $\Gamma$ in $\Diff_+^{>1}(\S1)$:
$$
DT(\Gamma)=\Diff_+^{>1}(\S1) \backslash 
\{f \in \QS \mid f\Gamma f^{-1} \subset \Diff_+^{>1}(\S1)\} \subset DT.
$$
Clearly, $AT(\Gamma)$ is closed in $AT$ and $DT(\Gamma)$ is closed in $DT$.

We consider the canonical projections
\begin{align*}
& \alpha:T=\Mob(\S1) \backslash \QS \longrightarrow AT=\Sym \backslash \QS;\\
& \theta:DT=\Diff_+^{>1}(\S1) \backslash \QS \longrightarrow AT=\Sym \backslash \QS.
\end{align*}
We note that $\alpha|_{T(\Gamma)}$ is not surjective.
More precisely, 
if $T(\Gamma) \neq \{[\id]\}$, then 
$\alpha T(\Gamma)$ is strictly contained in $AT(\Gamma)$ (see \cite[Theorem 1.1]{Mat6}).

On the other hand, the following rigidity theorems are proved in \cite[Theorems 2.2 and 4.1]{Mat7}.

\begin{theorem}\label{rigidity}
Let $\Gamma$ be a subgroup of $\Mob(\S1)$ containing a hyperbolic element.
\begin{enumerate}
\item
If $f \in \Sym$ satisfies $f \Gamma f^{-1} \subset \Mob(\S1)$, then $f \in \Mob(\S1)$.
\item
If $f \in \Sym$ satisfies $f \Gamma f^{-1} \subset \Diff_+^{1+\alpha}(\S1)$, then $f \in \Diff_+^{1+\alpha}(\S1)$.
\end{enumerate}
\end{theorem}

Theorem \ref{rigidity} (1) implies 
that the restriction
$\alpha|_{T(\Gamma)}:T(\Gamma) \to AT(\Gamma)$ is injective for a non-elementary Fuchsian group $\Gamma$.
As an application of 
Theorem \ref{rigidity} (2) and Theorem \ref {compatible3},
we will prove below that
$\theta|_{DT(\Gamma)}:DT(\Gamma) \to AT(\Gamma)$ is also injective. 
Then, under the identification of the embedded images in $AT(\Gamma)$ by $\alpha$ and $\theta$, we have
$$
T(\Gamma) \subset DT(\Gamma) \subset AT(\Gamma).
$$
We may ask a problem of which (or both) inclusion is strict.

\begin{theorem}\label{injective}
For a 
Fuchsian group $\Gamma \subset \Mob(\S1)$ with a hyperbolic element, the restriction of the projection 
$$
\theta|_{DT(\Gamma)}:DT(\Gamma) \to AT(\Gamma)
$$
is injective. Hence,
the deformation space $DT(\Gamma)$ can be realized in 
$AT(\Gamma)$. 
\end{theorem}

\begin{proof}
Suppose that there are $f_1,f_2 \in \QS$ such that (1) both $f_1\Gamma f_1^{-1}$ and $f_2\Gamma f_2^{-1}$
are subgroups of $\Diff_+^{>1}(\S1)$; (2) $f_2 \circ f_1^{-1} \in \Sym$ but 
$f_2 \circ f_1^{-1} \notin \Diff_+^{>1}(\S1)$. For $[f_1],[f_2] \in T$, we consider the Bers embeddings
$\varphi_1=\beta([f_1])$ and $\varphi_2=\beta([f_2])$ in $B(\D^*)$.
By Theorem \ref{compatible3}, condition (1) can be read as the orbit of $\varphi_1$ under $\Gamma$ is in 
$B_0^{>0}(\D^*)+\varphi_1$ and the orbit of $\varphi_2$ under $\Gamma$ is in 
$B_0^{>0}(\D^*)+\varphi_2$. Condition (2) can be read as $\varphi_1-\varphi_2$ is in $B_0(\D^*)$ but not in $B_0^{>0}(\D^*)$
(we also use Theorem \ref{B_0-B}).

We set $\varphi=t(\varphi_1-\varphi_2)$ for a sufficiently small constant $t>0$ with $\varphi \in \beta(T)$.
Since $\Gamma$ acts on $B(\D^*)$ linearly, the above conditions imply that the orbit of $\varphi$ under $\Gamma$ is
in $B_0^{>0}(\D^*)+\varphi$ and that $\varphi$ is in $B_0(\D^*)-B_0^{>0}(\D^*)$.
If we choose $f \in \Sym$ such that $\beta([f])=\varphi$, then these conditions are equivalent to that
$f \Gamma f^{-1} \subset \Diff_+^{>1}(\S1)$ and $f \notin \Diff_+^{>1}(\S1)$.
Here, by choosing any hyperbolic element $\gamma \in \Gamma$, we can find some $\alpha \in (0,1)$ such that
$f \langle \gamma \rangle f^{-1} \subset \Diff_+^{1+\alpha}(\S1)$. However, this contradicts Theorem \ref{rigidity} (2).
Therefore, there are no such $f_1,f_2$ satisfying the conditions mentioned at the beginning,
which shows that $\theta|_{DT(\Gamma)}$ is injective.
\end{proof}

We have handled so far the class $\Diff_+^{1+\alpha}(\S1)$ for $\alpha \in (0,1)$, but in the following result,
we can also consider higher regularity of circle diffeomorphisms at the same time.
We denote the group of such circle diffeomorphisms by $\Diff_+^{r}(\S1)$ for $r>1$.
The following corollary to the above theorem asserts that if we restrict 
a group of circle homeomorphisms to the one obtained by quasisymmetric conjugation of a M\"obius group,
we can extend the rigidity theorem from a diffeomorphic conjugation to a symmetric conjugation.
Here, we refer to a quasisymmetric conjugation of a M\"obius group
as a {\it uniformly quasisymmetric group}. Justification of using this terminology stems from the result of Markovic \cite{Mar}.

\begin{corollary}\label{extension2}
Let $\Gamma_1$ and $\Gamma_2$ be non-abelian uniformly quasisymmetric subgroups of
$\Diff^{r}_+(\S1)$ $(r >1)$.
If $f\Gamma_1 f^{-1}=\Gamma_2$ for $f \in \Sym$, then $f \in \Diff^{r}_+(\S1)$.
\end{corollary}

\begin{proof}
We may assume that $\Gamma_1=f_1 \Gamma f_1^{-1}$ and $\Gamma_2=f_2 \Gamma f_2^{-1}$
for a subgroup $\Gamma$ of $\Mob(\S1)$ with a hyperbolic element and that $f=f_2 \circ f_1^{-1} \in \Sym$.
From this, we see that
$f_1$ and $f_2$ modulo $\Diff_+^{>1}(\S1)$ belong to $DT(\Gamma)$. On the other hand, the assumption 
$f_2 \circ f_1^{-1} \in \Sym$ implies that the projections of these points in $DT(\Gamma)$ by $\theta$
are the same point of $AT(\Gamma)$. By the injectivity due to Theorem \ref{injective}, we obtain that 
$f_1$ coincides with $f_2$ modulo $\Diff_+^{>1}(\S1)$. In particular, $f=f_2 \circ f_1^{-1}$ is a diffeomorphism.
Then, by Theorem \ref{promotion} below, $f$ belongs to $\Diff^{r}_+(\S1)$.
\end{proof}

The following result is a special case of the theorem due to Ghys and Tsuboi \cite{GT}.
In fact, as
remarked by Navas \cite[p.152]{Nav}, their theorem can be generalized to $\Diff_+^{1+\alpha}(\S1)$
by showing
the Sternberg linearization theorem in the corresponding setting. 
A proof is given in the appendix for the sake of convenience.

\begin{theorem}\label{promotion}
Let $\Gamma$ be a subgroup of $\Mob(\S1)$ with a hyperbolic element. If $f \in \Diff_+(\S1)$
satisfies $f \Gamma f^{-1} \in \Diff_+^r(\S1)$ for $r>1$, then $f \in \Diff_+^r(\S1)$.
\end{theorem}
\medskip

We note that a certain part of the arguments in \cite[Section 4]{Mat7} for the proof of Theorem \ref{rigidity} (2)
can be replaced with this theorem.

\section{Appendix: The Sternberg linearization theorem}\label{8}

In this appendix, we give a proof of Theorem \ref{promotion}. To this end,
following the instruction by Navas \cite[p.150]{Nav}, 
we will do an exercise in proving a $C^{1+\alpha}$-version of
the Sternberg linearization theorem. We note that \cite[Theorem 3.6.2]{Nav} handles the case of $C^{r}$
for any integer $r \geq 2$ including $r=\infty$, and the non-integer case can be shown similarly to the argument below.

\begin{theorem}\label{Sternberg}
Let $g$ be a real-valued orientation-preserving
$C^{1+\alpha}$-diffeomorphism defined on some neighborhood of $0 \in \R$ such that $g(0)=0$ and
$g'(0)=a \neq 1$. Then, there exists a real-valued orientation-preserving
$C^{1+\alpha}$-diffeomorphism $h$ defined on some neighborhood of $0 \in \R$ with $h(0)=0$ and $h'(0)=1$
such that $h(g(x))=a h(x)$. 
Moreover, such an $h$ is unique. More precisely, if a real-valued orientation-preserving
$C^{1}$-diffeomorphism $h_1$ defined on some neighborhood of $0 \in \R$ satisfies the same properties
as $h$, then $h_1$ coincides with $h$ in some neighborhood of $0 \in \R$.
\end{theorem}

\begin{proof}
By considering $g^{-1}$ if necessary, we may assume that $0<a<1$. For $\delta \in (0,1)$, we
assume that $g$ is defined on $[-\delta,\delta]$. We set
$$
c_\delta=\sup_{-\delta \leq x,\,y \leq \delta} \frac{|g'(x)-g'(y)|}{|x-y|^\alpha}.
$$
Then, $c_\delta$ is decreasing as $\delta \to 0$. By setting $y=0$, we in particular have
$$
|g'(x)-g'(0)| \leq c_\delta |x|^\alpha \leq c_\delta \delta^\alpha \quad (-\delta \leq x \leq \delta).
$$
Therefore, $|g'(x)| \leq a+c_\delta \delta^\alpha$ and $|g(x)| \leq \delta(a+c_\delta \delta^\alpha)$.

Let $E_\delta$ be a linear space of real-valued $C^{1+\alpha}$-functions $\psi$ on $[-\delta,\delta]$
such that $\psi(0)=\psi'(0)=0$. By providing a norm
$$
\Vert \psi \Vert=\sup_{-\delta \leq x,\,y \leq \delta} \frac{|\psi'(x)-\psi'(y)|}{|x-y|^\alpha},
$$
we see that $E_\delta$ is a Banach space. As before, $|\psi'(x)| \leq \Vert \psi \Vert \delta^\alpha$ for
every $x \in [-\delta,\delta]$.
We define a linear operator $S_\delta:E_\delta \to E_\delta$ by
$$
S_\delta=\frac{1}{a}\, \psi \circ g.
$$

We will show that this is well-defined and the operator norm satisfies $\Vert S_\delta \Vert<1$ if
we choose a sufficiently small $\delta>0$. Since
\begin{align*}
& |(\psi \circ g)'(x)-(\psi \circ g)'(y)|\\
=&|\psi'(g(x))g'(x)-\psi'(g(y))g'(x)+\psi'(g(y))g'(x)-\psi'(g(y))g'(y)|\\
\leq& |\psi'(g(x))-\psi'(g(y))| \cdot |g'(x)|+|\psi'(g(y))| \cdot |g'(x)-g'(y)|,
\end{align*}
we have that
\begin{align*}
& \frac{|(\psi \circ g)'(x)-(\psi \circ g)'(y)|}{|x-y|^\alpha}\\
\leq& \frac{|\psi'(g(x))-\psi'(g(y))|}{|g(x)-g(y)|^\alpha} \cdot \frac{|g(x)-g(y)|^\alpha}{|x-y|^\alpha} \cdot|g'(x)|+
|\psi'(g(y))| \cdot \frac{|g'(x)-g'(y)|}{|x-y|^\alpha}.
\end{align*}
We choose $\delta>0$ so small that both $g(x)$ and $g(y)$ are in $[-\delta,\delta]$. Then,
the last term in the above inequality is bounded by
$$
\Vert \psi \Vert |g'(\xi)|^\alpha |g'(x)|+\Vert \psi \Vert \delta^\alpha c_\delta
\leq \Vert \psi \Vert \{(a+c_\delta \delta^\alpha)^{1+\alpha}+\delta^\alpha c_\delta \},
$$
where $\xi$ is some real number between $x$ and $y$. Therefore,
$$
\Vert S_\delta \Vert \leq \frac{1}{a} \{(a+c_\delta \delta^\alpha)^{1+\alpha}+c_\delta \delta^\alpha  \},
$$
which can be made smaller than $1$ by $\delta \to 0$ and hence $c_\delta \delta^\alpha \to 0$.

We set $\psi_*(x)=g(x)-ax$, which belongs to $E_\delta$. By fixing $\psi_*$, we consider
a functional equation
$$
S_\delta(\psi)+a^{-1} \psi_*=\psi
$$
with respect to $\psi \in E_\delta$. If we set the left side term as 
$$
F(\psi)=S_\delta (\psi)+a^{-1} \psi_*,
$$
then $F:E_\delta \to E_\delta$ satisfies
$$
\Vert F(\psi_1)-F(\psi_2) \Vert=\Vert S_\delta(\psi_1-\psi_2) \Vert \leq \Vert S_\delta \Vert \Vert \psi_1-\psi_2 \Vert
$$
with $\Vert S_\delta \Vert<1$. By the Banach contraction principle,
there exists $\psi_0 \in E_\delta$ uniquely such that $F(\psi_0)=\psi_0$.

We define $h(x)=x+\psi_0(x)$. This satisfies that
\begin{align*}
h(g(x))&=g(x)+\psi_0 \circ g(x)
=\psi_*(x)+ax+\psi_0 \circ g(x)\\
&=\psi_*(x)+ax+a S_\delta(\psi_0)(x)
=a\psi_0(x)+ax=ah(x).
\end{align*}
Thus, this function $h$ is the desired one.

Next, we show the latter statement on the uniqueness.
For any $x$ in some neighborhood of $0 \in \R$, we have
$$
ah \circ h_1^{-1}(x)=h \circ g \circ h_1^{-1}(x)=h \circ h_1^{-1}(ax).
$$
Then, it follows that
$$
h \circ h_1^{-1}(x)=\frac{h \circ h_1^{-1}(a^n x)}{a^n}=x\, \frac{h \circ h_1^{-1}(a^n x)}{a^n x} 
$$
for any $n \in \N$, and this tends to $x$ as $n \to \infty$.
Hence, $h \circ h_1^{-1}(x)=x$, that is, $h(x)=h_1(x)$.
\end{proof}

We have the following consequence in the same way as \cite[Corollary 3.6.3]{Nav}.

\begin{corollary}\label{coro}
For $r>1$, let $g_1$ and $g_2$
be $C^{r}$-diffeomorphisms satisfying all properties of $g$ as in Theorem \ref{Sternberg}.
If a $C^{1}$-diffeomorphism $\varphi$ defined on some neighborhood of $0 \in \R$ with $\varphi(0)=0$ 
conjugates $g_1$ to $g_2$, then $\varphi$ is a $C^{r}$-diffeomorphism.
\end{corollary}

\medskip
\noindent
{\it Proof of Theorem \ref{promotion}.}
For a hyperbolic element $\gamma \in \Gamma$, we set $g=f\gamma f^{-1} \in \Diff^r_+(\S1)$.
For the attracting fixed point $\xi \in \S1$ of $\gamma$, $f(\xi)$ is the attracting fixed point of $g$ and 
$|g'(f(\xi))|=|\gamma'(\xi)| \neq 1$.
In local coordinates around $f(\xi) \in \S1$, $g$ is represented 
by a real-valued orientation-preserving
$C^{r}$-diffeomorphism $\widetilde g$
defined on some neighborhood of $0 \in \R$ such that $\widetilde g(0)=0$ and $\widetilde g'(0)=a \neq 1$.
We also consider the representation of $\gamma$ in local coordinates around $\xi \in \S1$, which is
also a real-valued orientation-preserving
$C^{r}$-diffeomorphism $\widetilde \gamma$
defined on some neighborhood of $0 \in \R$ such that $\widetilde \gamma(0)=0$ and $\widetilde \gamma'(0)=a \neq 1$.
Then, a $C^1$-diffeomorphism $\varphi$ between these neighborhoods, which stems from $f$, conjugates
$\widetilde \gamma$ to $\widetilde g$. By Corollary \ref{coro}, we see that $\varphi$ is a $C^{r}$-diffeomorphism.

By the above argument, we see that $f$ is in $C^{r}$ in some neighborhood of $\xi \in \S1$.
However, the iteration of $\gamma^{-1}$ expands this neighborhood to $\S1$ except the repelling fixed point of
$\gamma$, and we find that $f$ is in $C^{r}$ there. Finally, by exchanging the roles of
the attracting and the repelling fixed points, 
we conclude that $f$ is a $C^{r}$-diffeomorphism of the entire space $\S1$.
\qed

\end{document}